\documentclass{article}
\usepackage{fullpage}
\usepackage{graphicx}

\usepackage{ bbm, color,amsfonts, amsmath, amssymb,amsthm}
\usepackage{hyperref}
\usepackage{cleveref}

\usepackage[small]{caption}
\usepackage{comment} 
 \usepackage[algo2e]{algorithm2e} 
\usepackage{algorithm}
\usepackage{thmtools}
\usepackage{thm-restate}
\usepackage{fancybox}

\newcommand{\defeq}{:=}
\newcommand{\argmin}{\mathrm{argmin}}

\newcommand{\diag}[1]{{\bf Diag}\left({#1}\right)}

\newcommand{\norm}[1]{\ensuremath{\left\lVert #1 \right\rVert}}

\newcommand\nat{\mathbb N}
\newcommand\rea{\mathbb R}
\newcommand{\R}{\rea}

\newcommand\poly{\mathrm{poly}}  %usage \poly(n)

\newcommand\calA{\mathcal{A}}

\newcommand\calE{\mathcal{E}}

\newcommand\calO{\mathcal{O}}

\newcommand\calT{\mathcal{T}}

\newtheorem{theorem}{Theorem}[section]

\newtheorem{lemma}[theorem]{Lemma}
\newtheorem{corollary}[theorem]{Corollary}

\newtheorem{fact}[theorem]{Fact}

\newtheorem{definition}[theorem]{Definition}

\let\originalleft\left
\let\originalright\right
\renewcommand{\left}{\mathopen{}\mathclose\bgroup\originalleft}
  \renewcommand{\right}{\aftergroup\egroup\originalright}

\def\setof#1{\left\{#1  \right\}}

\def\diag#1{\textsc{Diag}\left( #1 \right)}

\newcommand\bb{b}
\newcommand\cc{c}
\newcommand\dd{d}
\newcommand\ee{\boldsymbol{\mathit{e}}}
\newcommand\ff{f}

\renewcommand\gg{g}

\newcommand\vv{\boldsymbol{\mathit{v}}}
\newcommand\ww{\boldsymbol{\mathit{w}}}
\newcommand\yy{y}
\newcommand\zz{z}
\newcommand\xx{x}

\newcommand\hh{h}

\newcommand\xxtil{\tilde{x}}

\renewcommand\AA{A}

\newcommand\CC{C}

\newcommand\RR{R}

\newcommand\WW{W}

\newcommand\Otil{\widetilde{O}}

 {
	\begin{enumerate}}{\end{enumerate}}

\global\long\def\eps{\epsilon}

\newcommand\supp[1]{\mathrm{supp}\left\{#1\right\}}

\newcommand\Taylor{\calT_f^{s-1}}

\newcommand{\inner}[2]{\left\langle #1,\, #2 \right\rangle}

\newcommand{\sigmaMS}{\sigma$-$\textrm{\emph{MS}}_p}
\newcommand{\GMS}{\textrm{\emph{GMS}}_p}

\usepackage[
  backend=biber,
  backref=true,
  backrefstyle=none,
  date=year,
  doi=false,
  giveninits=true,
  hyperref=true,
  maxbibnames=10,
  sortcites=false,
  style=alphabetic,
  url=false, 
]{biblatex}

\bibliography{refs}

\title{Convex optimization with $p$-norm oracles}

\author{Deeksha Adil\\Institute for Theoretical Studies\\
ETH Zürich\\deeksha.adil@eth-its@ethz.ch \and Brian Bullins\\
Department of Computer Science\\
Purdue University\\ bbullins@purdue.edu \and Arun Jambulapati\\Independent Researcher\\jmblpati@gmail.com \and Aaron Sidford\\ Department of Computer Science\\Stanford University\\sidford@stanford.edu }

\begin{document}

\maketitle

\begin{abstract}
 In recent years, there have been significant advances in efficiently solving $\ell_s$-regression using linear system solvers and $\ell_2$-regression [Adil-Kyng-Peng-Sachdeva, J. ACM'24]. Would efficient smoothed $\ell_p$-norm solvers lead to even faster rates for solving $\ell_s$-regression when $2 \leq p < s$?
In this paper, we give an affirmative answer to this question and show how to solve $\ell_s$-regression using 
$\tilde{O}(n^{\frac{\nu}{1+\nu}})$ iterations of solving smoothed $\ell_p$ regression problems, where $\nu := \frac{1}{p} - \frac{1}{s}$.
To obtain this result, we provide improved accelerated rates for convex optimization problems when given access to an \emph{$\ell_p^s(\lambda)$-proximal oracle}, which, for a point $\cc$, returns the solution of the regularized problem 
$\min_{\xx} \ff(\xx) + \lambda \|\xx-\cc\|_p^s$. Additionally, we show that these rates for the $\ell_p^s(\lambda)$-proximal oracle 
are optimal for algorithms that query in the span of the outputs of the oracle, and we further apply our techniques to settings of high-order and quasi-self-concordant optimization.
\end{abstract}

 \newpage
\tableofcontents

\newpage

\section{Introduction}

A prominent approach to efficiently solving convex optimization problems is to reduce them to solving a sequence of linear systems. For example, interior point methods reduce linear programming, semidefinite programming, many empirical risk minimization problems, and more to essentially solving a sequence of \emph{$\ell_2$ (or least-squares) regression problems} or, more broadly, to a sequence of linear systems~\cite{karmarkar1984new, lee2019solving, nesterov1994interior, renegar1988polynomial, wright1997primal}. More recently, improvements in quasi-self-concordant~\cite{bach2010self,carmon2020acceleration,karimireddy2018global} and high-order~\cite{bullins2020highly,nesterov2006cubic, nesterov2021implementable} optimization have yielded new ways of using linear system solvers for solving well-established machine learning problems, including $\ell_\infty$ and logistic regression.

Similarly, recent advances in \emph{iterative refinement} have shown how to reduce $\ell_p$-regression, a natural data analysis problem, to solving a sequence of linear systems~\cite{adil2019iterative,adil2022fast}.  $\ell_p$ regression problems appear in a wide range of applications, including semi-supervised learning~\cite{alamgir2011phase, flores2022analysis, kyng2015algorithms}, data clustering and learning~\cite{elmoataz2015p,elmoataz2017game}, and low-rank matrix approximations~\cite{chierichetti2017algorithms}.

In this paper, we study a natural generalization: rather than reducing to $\ell_2$-regression, i.e.,
\[
\min_{\xx \in \rea^d}\|\AA\xx-\bb\|_2^2,
\]
we consider reductions to \emph{smoothed} $\ell_p$-regression, i.e., problems of the form
\[
\min_{\xx \in \rea^d}\gg^{\top}\xx + \|\AA\xx-\bb\|_2^2 + \|\CC\xx-\dd\|_p^p\,.
\]
We ask: \emph{would smoothed $\ell_p$-regression solvers yield faster rates for solving optimization problems?}

Addressing this question could advance the foundations of optimization theory and provide new tools for structured and high-order optimization beyond $\ell_p$-regression. For example, as we will see, this question has connections to highly-smooth optimization in different norms, ball-acceleration in different geometries, and quasi-self-concordant optimization.

\paragraph{Efficient reductions from $\ell_s$- to $\ell_p$-regression.}
We consider the particular problem of efficiently reducing $\ell_s$-regression to smoothed $\ell_p$-regression for $2 \leq p < s$. We provide a new efficient reduction that solves this problem. Moreover, we obtain this result by an optimization theoretic framework of broader utility, which indeed has implications for structured optimization and higher-order optimization.
We formally define $\ell_s$ regression in Definition~\ref{def:regression} and then provide Theorem~\ref{thm:regression-main}, which contains our main result for this problem (proved in Section~\ref{sec:regression}).

\begin{definition}[$\ell_s$-Regression]
\label{def:regression}
    Given $2\leq s<\infty$, $\epsilon>0$, $\AA\in \mathbb{R}^{n \times d}$, and $\bb\in \mathbb{R}^n$, the \emph{$\ell_s$-regression problem} is to find $\widetilde{\xx}\in \rea^d$ such that
\begin{equation}\label{eq:sReg}
   \|\AA\widetilde{\xx}-\bb\|_s^s \leq (1+\epsilon)\min_{\xx\in \rea^d}  \|\AA\xx-\bb\|^s_s.
\end{equation}
We call such an $\tilde{x}$ an $\epsilon$-approximate solution to the problem.
\end{definition}

\begin{restatable}[From $\ell_s$-Regression to Smoothed $\ell_p$-Regression]{theorem}{SNormReg}
\label{thm:regression-main}
There is an algorithm that, given  $\epsilon > 0$, $s > p \geq 2$, computes an $\epsilon$-approximate solution to $\ell_s$-regression (Definition~\ref{def:regression}) in at most $O( s \cdot n^{\frac{\nu}{1+\nu}}\log^2\frac{n}{\epsilon})$ iterations for $\nu \defeq \frac{1}{p} - \frac{1}{s}$, each of which can be implemented in $O(d)$ time plus the time needed to solve $\tilde{O}(1)$\footnote{We use $\Tilde{O}$ to hide $p,s$, and $\poly\log n$ factors.} smoothed $\ell_p$-regression problems of the form
\[
\min_{\xx \in \rea^d}\gg^{\top}\xx + \|\tilde{\AA}\xx-\tilde{\bb}\|_2^2 + \|\CC\xx-\dd\|_p^p\,.
\]
\end{restatable}

Prior reductions include one of \cite{bubeck2018homotopy} which established that $\tilde{O}(n^{\frac{1}{2}-\frac{1}{p}})$ iterations of comparable cost sufficed for $p = 2$. This was then improved by~\cite{adil2019iterative} to $\tilde{O}(n^{\frac{p-2}{3p-2}})$, i.e., by a factor of $\Omega(n^{\frac{(p-2)^2}{2p(3p-2)}})$. For a general $p$, the previous state-of-the-art for such a reduction is due to \cite{adil2020faster}. Though not explicitly shown, their results imply an analogous \Cref{thm:regression-main} with an iteration bound of $\tilde{O}(n^{\frac{s\nu}{s-1}})$.

Our work improves upon this rate for all $2\leq p<s$ by a factor of $\Omega(n^{\frac{\nu^2}{1+\nu}})$ (similar to the prior improvement of \cite{adil2019iterative} over \cite{bubeck2018homotopy} for $p = 2$). 

\paragraph{Optimal acceleration of $\ell_p^s(\lambda)$-proximal oracles.}

To prove \Cref{thm:regression-main}, we consider a general optimization problem of independent interest. Specifically, we consider the problem of minimizing a convex function $\ff$ given (approximate) access to what we call an \emph{$\ell_p^s(\lambda)$-proximal oracle}.

\begin{definition}[$\ell_p^s(\lambda)$-Proximal Oracle]\label{def:lpsProx} For $p, s \geq 2$, $\lambda > 0$, an \emph{$\ell_p^s(\lambda)$-proximal oracle} for $\ff:\R^d\rightarrow \R$ is an oracle that, when queried at a ``centering point'' $\cc\in \R^d$, returns 
\[
\tilde{x} \in \argmin_{\xx \in \R^d} \ff(\xx) + \lambda \|\xx-\cc\|_p^s.
\]
\end{definition}

This problem is well-studied in the special case when $p = s = 2$, where it corresponds to the standard (quadratic) proximal oracle~\cite{parikh2014proximal}. In particular, algorithms that access these oracles have been studied in convex optimization theory (see e.g.,~\cite{boyd2011distributed, combettes2007douglas, rockafellar1976monotone}). A notable use of this oracle is in the \emph{accelerated proximal point method}~\cite{guler1992new, nesterov1983method}, which computes an $\epsilon$-optimal solution with $O(\lambda^{1/2}\norm{\xx_0-\xx^*}_2 \eps^{-1/2})$ queries to the oracle, where $\xx^* \in \R^d$ is used to denote the minimizer of $\ff$ throughout the paper. This query complexity is optimal~\cite{nemirovskij1983problem}, and the method has played an important role in numerous algorithmic advances~\cite{frostig2015regularizing, lin2015universal, parikh2014proximal, schmidt2017minimizing, shalev2014accelerated}, as well as in the high-order acceleration frameworks of \cite{monteiro2013accelerated, gasnikov2019near, carmon2022optimal}, and the ball acceleration framework of \cite{carmon2020acceleration} (which has multiple applications \cite{carmon2021thinking, carmon2023resqueing, carmon2024whole, jambulapati2024closing}).

However, outside of the above examples, optimal rates for solving convex optimization with $\ell_p^s(\lambda)$-proximal oracles are not known, to the best of our knowledge. This is particularly relevant for $\ell_p$ regression as, building upon an approach of~\cite{adil2022fast}, solving $\ell_s$-regression using the smoothed $\ell_p$ oracle reduces to acceleration with our $\ell_p^s(\lambda)$-proximal oracle.
Consequently, to prove \Cref{thm:regression-main} we give new efficient algorithms (Algorithm~\ref{alg:LSFProx}), for which we prove the following.
\begin{theorem}[Accelerated Optimization with $\ell_p^s(\lambda)$-Proximal Oracle]\label{thm:ConceptualProx-main}
   Algorithm~\ref{alg:LSFProx} given $s > p \geq 2$, $\epsilon > 0$, $\xx_0 \in \R^d$, outputs $\xxtil \in \R^d$ with $\ff(\xxtil)- \ff(\xx^{\star})\leq \epsilon$ using $O((s\lambda \|\xx_0-\xx^{\star}\|_p^s / \epsilon)^{\frac{1}{s(1+\nu)}})$ queries to an $\ell_p^s(\lambda)$-proximal oracle. Each iteration can be implemented in $O(d)$ time plus time of the $\ell_p^s(\lambda)$-proximal oracle.
\end{theorem}

Though the rates of Theorem~\ref{thm:ConceptualProx-main} may look unnatural at first glance, we in fact prove that, for any $p, s \geq 2$, they are optimal for broad classes of algorithms! Specifically, we show (Section~\ref{sec:LB}) that there exists a function such that any zero-respecting algorithm, i.e., one that at every iteration queries an $\ell_p^s(\lambda)$-proximal oracle centered at a point $\cc \in \R^d$ which lies in the span of the outputs of the oracle, requires at least $\Omega((\lambda \|\xx^{\star}\|_p^s/\epsilon)^{\frac{1}{s(1+\nu)}})$ iterations to return an $\epsilon$-suboptimal point.

Note that there is a broader literature of lower bounds for related problems, e.g., convex optimization in non-Euclidean norms with gradient oracles~\cite{guzman2015lower} and in parallel settings~\cite{diakonikolas2019lower}. However, we are are unaware of such lower bounds for our particular $\ell_p^s(\lambda)$-proximal oracle.

\begin{restatable}[$\ell_p^s(\lambda)$-Proximal Oracle Optimization Lower Bound]{theorem}{LowerBound}\label{thm:LowerBound-main}
     For every $\ell_p^s(\lambda)$-proximal zero-respecting algorithm (Definition~\ref{def:zero-resp-alg}) given $p,s \geq 2$, $\lambda > 0$, $\epsilon > 0$, there is a function $\ff : \rea^d \rightarrow \rea$ such that any $\xx$ in the first $k = O((\lambda \|\xx^{\star}\|_p^s / \epsilon)^{\frac{1}{s(1+\nu)}})$ iterations of the algorithm satisfies $
    \ff(\xx)-\ff(\xx^{\star})\geq \epsilon$.
\end{restatable}

\paragraph{Optimal acceleration of non-Euclidean ball-constrained optimization oracles.} 
We show how our approach can also be used to show lower bounds for what we call \emph{$\ell_p$ ball-constrained oracles}.

\begin{definition}[$\ell_p$ Ball-constrained Oracle]\label{def:ball} For $p \geq 2$, $r > 0$, an $\ell_p^{\infty}(r)$-proximal oracle (also referred to as an \emph{$\ell_p$ ball-constrained oracle}) for $\ff:\R^d\rightarrow \R$ is an oracle that, when queried at a centering point $\cc\in \R^d$, returns
\[
\tilde{x} \in \argmin_{x \in \R^d | \|\xx-\cc\|_p\leq r}\ \ff(\xx).
\]
\end{definition}

We extend previous $\ell_2$ ball oracle results~\cite{carmon2020acceleration, carmon2022optimal, karimireddy2018global}
to the $\ell_p$ setting (Section~\ref{sec:NoLS}), and we also show that our rates are optimal for zero-respecting algorithms (Section~\ref{sec:LB}).

\begin{theorem}[Accelerated Optimization with $\ell_p$ Ball-constrained Oracle]\label{thm:Ball-main}
   Algorithm~\ref{alg:LSFProx} given $p \geq 2$, $r > 0$, $\epsilon > 0$, $\xx_0 \in \R^d$, outputs $\xxtil$ such that $\ff(\xxtil)- \ff(\xx^{\star})\leq \epsilon$ using  
   \[
   O( (p\|\xx_0-\xx^{\star}\|_p /r)^{\frac{p}{p+1}}\log(\|\xx_0-\xx^{\star}\|_p^p/\epsilon))
   \]
    queries to an $\ell_p^{\infty}(r)$-proximal oracle. Each iteration can be implemented in $O(d)$ time plus time of the $\ell_p^{\infty}(r)$-proximal oracle.
\end{theorem}

\begin{restatable}[$\ell_p$ Ball-constrained Oracle Optimization Lower Bound]{theorem}{LowerBoundInf}\label{thm:LowerBoundInf-main}
    For every $\ell_p^{\infty}(r)$-proximal zero-respecting algorithm (Definition~\ref{def:infProxAlg}) given $p \geq 2$, $r > 0$, there is a function $\ff : \rea^d \rightarrow \rea$ such that any $\xx$ in the first $k = O((\|\xx^{\star}\|_p / r )^{\frac{p}{p+1}})$ iterations of the algorithm satisfies 
    \[
    \ff(\xx)-\ff(\xx^{\star})\geq \Omega(\|\xx^{\star}\|_p^{1/(p+1)}r^{p/(p+1)})\,.
    \]
\end{restatable}
As previously mentioned, algorithms based on $\ell_2$ ball-constrained oracles have had a wide range of applications. For example, they have arisen in solving natural problems in learning theory
and machine learning, including logistic regression~\cite{karimireddy2018global,carmon2020acceleration}, minimizing the maximum of a set of
convex losses~\cite{carmon2021thinking, carmon2024whole}, and parallel stochastic optimization~\cite{bubeck2019complexity, carmon2023resqueing}. Consequently, these results could
lead to opportunities for improving these and related problems in different geometries for, e.g., high-order and quasi-self-concordant optimization settings. We further discuss in Section~\ref{sec:App} how our methods can apply to such cases by showing that the oracles for these settings implement a more general oracle class.

\paragraph{Extensions to high-order optimization.} 
As another application, we provide the following result, which extends our techniques to \emph{highly smooth} problems with respect to the $p$-norm, i.e., problems whose $q^{th}$-order derivatives are $L_q$-Lipschitz continuous (for $q \geq 1$), in which case each step is based on minimizing a $\norm{\cdot}_p^{q+1}$-regularized $q^{th}$-order Taylor expansion.\footnote{We use the notation $O_{p,s}(\cdot)$ to hide multiplicative running time factors which depend only on $p$ and $s$.}
\begin{restatable}[High-order Optimization]{theorem}{thmHO}\label{thm:HO}
    For $s > p \geq 2$, $q = s-1$, let $\ff$ be $q^{th}$-order $L_{q}$ smooth with respect to $\norm{\cdot}_p$, $\xx_0 \in \rea^d$, and $\epsilon > 0$. Algorithm~\ref{alg:LSFProx}, implementing an $O(1)$-$\GMS$ oracle by a $(p,s)$-Taylor oracle, finds $\xxtil$ such that $\ff(\xxtil)-\ff(\xx^{\star})\leq \epsilon$ in 
    \[
    O_{p,s}\left(\norm{x_0 - x^*}_p^{\frac{1}{p(1+\nu)}} (L_{q}/\epsilon)^{\frac{1}{s(1+\nu)}}\right)
    \]iterations. The cost of each iteration is at most $O(d)$ plus the cost of the $(p,s)$-Taylor oracle.
\end{restatable}
We then show how this behaves, in a sense, as a certain \emph{approximation} to the $\ell_p^s(\lambda)$-proximal oracle (for an appropriate choice of $\lambda$, which depends on the smoothness parameter $L_{q}$, and where $s = q+1$), and the rates we obtain naturally generalize those found in previous results~\cite{gasnikov2019near, monteiro2013accelerated}. Such Euclidean-based high-order algorithms have been used, as previously discussed, to improve convergence for fundamental machine learning problems~\cite{agarwal2017finding,bullins2020highly,nesterov2006cubic}, and consequently rates for the non-Euclidean variants may have additional applications.

\paragraph{Simultaneous Independent Related Work.} Independently, \textcite{contreras2024non} also investigate non-Euclidean high-order and proximal point methods for convex optimization and provide similar algorithms and convergence rates. Specifically, they give algorithms for $q^{th}$-order $(L,\nu)$-Hölder continuous functions for all $p,q\geq 1$, whereas our work focuses on the case where $q +1 \geq p \geq 2$ for proximal oracles, ball-constrained optimization oracles, and highly smooth functions. \textcite{contreras2024non} further establish lower bounds for minimizing  $q$-th order $(L, \nu)$-Hölder continuous functions with respect to $||\cdot||_p$, for algorithms that can query a local oracle, and as a result, their lower bounds depend on $L$. In contrast, our lower bounds are for (zero-respecting) $\ell_p^s$ proximal algorithms, and they place no smoothness restrictions on the function class. The initial draft of this paper was produced independently without having seen \cite{contreras2024non}. Attempts were made to preserve independence during the revision process.

\paragraph{Future work and open problems.}
This paper provides improved rates for multiple convex optimization problems by using non-Euclidean optimization. However, these new rates come with costs. Whereas $\ell_2$-regression can be solved by matrix multiplication and linear system solving, smoothed $\ell_p$-regression is a more general and potentially complex optimization problem. Consequently, implementing steps of our algorithm efficiently and using this to obtain end-to-end runtime improvements for structured optimization problems is an interesting direction for future research.

Additionally, though we settle the complexity of convex optimization with the $\ell_p^s(\lambda)$-proximal oracle, this does not necessarily imply that our reduction from $\ell_s$ to smoothed $\ell_p$-regression is optimal. Consider, for example, solving $\ell_6$-regression using $\ell_2$-regression and compare it with solving $\ell_6$-regression via $\ell_4$-regression. Our results state that the former would use $\approx n^{1/4}$ iterations of solving least squares regression, while the iteration complexity of the latter would improve to $\approx n^{1/13}$ (improving upon $\approx n^{1/10}$ iteration complexity due to \cite{adil2020faster}), each of which would involve solving an $\ell_4$-regression problem. Curiously, if we further solve each $\ell_4$-regression problem using $\ell_2$-regression, we would require $n^{1/5}$ such problems at each step, leading to a total of $n^{\frac{1}{13} + \frac{1}{5}}\geq n^{\frac{1}{4}}$. Improving our reductions to smoothed $\ell_p$-regression, finding a similar type of reduction that doesn't occur this type of loss in repeated reduction, or providing more compelling evidence of optimality are additional interesting directions future research.

Furthermore, in the future, high-order optimization methods (e.g.~\cite{gasnikov2019near}) could potentially benefit from these tools, as such methods often depend on solving $\ell_p$-norm regularized subproblems at each iteration. While there has been a line of work showing how subproblems based on second-~\cite{nesterov2006cubic} and third-order~\cite{nesterov2021implementable} Taylor models admit more computationally efficient approximate solvers, less is known about more general settings. This work could potentially help in such settings.

Finally, obtaining similar improvements as the ones we have shown in this work for linear programming---or convex optimization with interior point methods more broadly---are additional exciting directions for future research. 
We hope our results help facilitate in these directions and, with further research, may yield new non-Euclidean perspectives on basic problems and algorithms in machine learning, including those related to parallel, stochastic, and higher-order optimization, and ultimately lead to further efficiency gains in theory and in practice.

\paragraph{Paper overview.}

In Section~\ref{sec:Prox} as a warm-up we give our main $\ell_p^s(\lambda)$-proximal point algorithm which includes a line-search in each iteration. In Section~\ref{sec:NoLS} we show how we may remove the need for this line-search and give an algorithm with a more general oracle. In Section~\ref{sec:App} we use these algorithms to give rates for solving $\ell_s$-regression and generalized high-order smooth problems. In Section~\ref{sec:LB}, we give lower bounds that match the rates of our $\ell_p^s(\lambda)$-proximal point algorithms analyzed in Section~\ref{sec:Prox}.

\section{Warm-up: Conceptual Proximal Point with Line Search}\label{sec:Prox}

In this section, we analyze a simpler version of our generalized proximal point algorithm. Our aim is first to provide clarity for the convergence rates, while in the following section we address the issue of overall computational cost. Specifically, in this section we present an algorithm where each iteration involves finding points that satisfy certain implicit conditions, similar to previous works based on Monteiro-Svaiter acceleration~\cite{monteiro2013accelerated}. These implicit problems can be solved using an additional line-search procedure as in previous works (e.g.,~\cite{monteiro2013accelerated,gasnikov2019near,carmon2020acceleration}), followed by invoking an $\ell_p^s(\lambda)$-proximal oracle. We later show, in Section~\ref{sec:NoLS}, how to modify the algorithm to obtain a line search-free method, and in doing so we introduce a more general oracle model. We further split our algorithm into two separate cases: $s < \infty$ and $s=\infty$ (the latter coinciding with an $\ell_p$ ball-constrained oracle).

Throughout the paper, we let $\norm{\cdot}_p$ (for $p \in [1,\infty)$) denote the standard $\ell_p$ norm, and we use $\omega_p$ to denote, for an initial point $x_0\in \rea^d$, the Bregman divergence of the $\ell_p$-norm function $\norm{x - x_0}_p^p$, i.e., for any $x,y\in \rea^d$, 
\[
\omega_p(\xx,\yy) = \|\xx-\xx_0\|_p^p - \|\yy-\xx_0\|_p^p  - \langle \nabla_{\yy}\|\yy-\xx_0\|_p^p, \xx-\yy\rangle. 
\]
We also let $p^* := \frac{p}{p-1}$ throughout the paper. 

We will prove the guarantees of Algorithm~\ref{alg:ProxQ}, which finds $x$ such that $f(x)-f(x^*)\leq \epsilon$ for a convex function $f$ via an $\ell_p^s(\lambda)$-proximal oracle given a line-search to estimate $\lambda_t$'s. Our algorithm, in every iteration maintains two points $x_t$ and $z_t$, takes a carefully chosen convex combination of the points to get $y_t$, applies the $\ell_p^s(\lambda)$-proximal oracle at $y_t$ to compute $x_{t+1}$, and then uses the gradient of $f$ at $x_{t+1}$ to compute $z_{t+1}$. The main guarantee of the algorithm is given below in Theorem~\ref{thm:ConceptualProx}. 

\begin{algorithm}
\caption{Conceptual Proximal Point Algorithm}
\label{alg:ProxQ}
  \textbf{Initialize:} $a_0 = 1$, $A_0 = 0$,
  $\zz_0 = \yy_0 =\xx_0$, $T \geq 1$\\
\For{$t = 0,\cdots, T-1$}{
Find $\lambda_t>0$, $\xx_{t+1}\in \rea^d$ for which the following hold:
\begin{itemize}
  \item $5^{p-1}\lambda_t a_{t+1}^p = A_{t+1}^{p-1}$, $A_{t+1} = A_t + a_{t+1}$, and $\yy_t = \frac{A_t}{A_{t+1}}\xx_t + \frac{a_{t+1}}{A_{t+1}}\zz_t$
 \item \small{ $\lambda_t = \begin{cases}
    \lambda\|\xx_{t+1}-\yy_t\|_p^{s-p} & \text{ if $s<\infty$}\\
      \|\nabla\ff(\xx_{t+1})\|_{\frac{p}{p-1}}/r^{p-1} & \text{ if $s = \infty$}
  \end{cases}$ and $\xx_{t+1} \gets \begin{cases} \arg\min_{\yy} \ff(\yy) + \frac{\lambda}{p}\|\yy-\yy_t\|_p^s & \text{ if $s<\infty$}\\
 \arg\min_{\norm{\yy - \yy_t}_p \leq r} \ff(\yy) & \text{ if $s = \infty$}\end{cases}$}
 \end{itemize}
 $\zz_{t+1} \gets \arg\min_{\zz\in \rea^d} \sum_{i\in [t+1]}a_i \nabla \ff(\xx_i)^{\top}(\zz-\yy_i) + \|\zz-\yy_0\|_p^p$
 }
 \Return $\xx_T$
\end{algorithm}

\begin{restatable}{theorem}{ThmConceptProx}\label{thm:ConceptualProx}
    Let $s<\infty$. Given $\ff$, $\xx_0 \in \R^d$, and $\epsilon > 0$, there exist, for all $t\geq 0$, $\lambda_t, x_{t+1}$ satisfying the conditions of Algorithm~\ref{alg:ProxQ}, and it outputs $\xx_T$ such that $f(\xx_T) - \ff(\xx^{\star} )\leq \epsilon$, in
    \[
    O\left(\frac{s^{s(1+\nu)}}{p^{s\nu}} \cdot \frac{\lambda\norm{x_0-x^*}^s_p}{\epsilon}\right)^{\frac{1}{s(1+\nu)}}
    \enspace\text{iterations.}
    \]
\end{restatable}

When $s = \infty$, i.e., the algorithm is equipped with an $\ell_p^{\infty}(r)$-proximal oracle, we prove the following.
\begin{theorem}\label{thm:BallOracle}
    Let $s=\infty$. Given $\ff,\xx_0 \in \R^d$ and $\epsilon > 0$, there exist, for all $t\geq 0$, $\lambda_t, x_{t+1}$ satisfying the conditions of Algorithm~\ref{alg:ProxQ}, and it outputs $\xx_T$ such that $
    f(\xx_T) - \ff(\xx^{\star} ) \leq \epsilon,$ in 
    \[
    O\left( \left(\frac{p \|\xx^{\star}-\xx_0\|_p}{r}\right)^{\frac{p}{p+1}}\log \left(\frac{\|\xx^{\star}-\xx_0\|_p}{\epsilon}\right)\right)
    \enspace\text{iterations.}
    \]
\end{theorem}

The crux of the proofs of Theorem~\ref{thm:ConceptualProx} and Theorem~\ref{thm:BallOracle} is a standard potential argument~\cite{diakonikolas2019approximate, nesterov2018lectures}. Namely, we rely on the following lemma, which shows how the potential $A_t(\ff(\xx_t) - \ff(\xx^{\star})) + \omega_p(\xx^{\star},\zz_{t})$ decreases in every iteration of the algorithm. The proof can be found in Appendix~\ref{app:concept}.
\begin{lemma}\label{lem:Potential}
   For all $t \geq 0$, the iterates of Algorithm~\ref{alg:ProxQ} satisfy
\begin{align*}
A_{t+1}\left(\ff(\xx_{t+1})-\ff(\xx^{\star})\right)  + \omega_p&(\xx^{\star},\zz_{t+1})\\& \leq A_{t}\left(\ff(\xx_{t})-\ff(\xx^{\star})\right) + \omega_p(\xx^{\star},\zz_{t}) - A_{t+1}\lambda_t \|\xx_{t+1}- \yy_t\|_p^p.
\end{align*}
\end{lemma}

The remaining proof of Theorem~\ref{thm:ConceptualProx} and Theorem~\ref{thm:BallOracle} follow slightly different pathways and can be found in Appendix~\ref{app:concept}.

\section{Line-Search Free Method for Non-Euclidean Acceleration}\label{sec:NoLS}

In this section, we describe our main algorithm which eliminates the line-search required when implementing Algorithm~\ref{alg:ProxQ} and works with a more general oracle. We first define our oracle.
\newcommand{\Oracle}{\mathcal{O}_{\sigma,p}}
\newcommand{\sigmaGMS}{\sigma$-$\textrm{\emph{GMS}}_p} 
\begin{definition}
     An oracle $\Oracle : \mathbb{R}^d \rightarrow \R^d \times \R_+$ is a \emph{$\sigma$-Generalized Monteiro-Svaiter} with respect to $p$ ($\sigmaGMS$) oracle for a function $f : \R^d \rightarrow \R$ and $\sigma \in [0,1)$ if, for every $\yy \in \R^d$, $(\xx, \gamma) = \Oracle(\yy)$
    satisfies 
    \begin{equation*}
         \langle\nabla \ff(\xx),\xx -\yy  \rangle
        \leq -(1-\sigma)\gamma\norm{\xx -\yy}_p^p, \quad \text{and}\quad \norm{\nabla \ff(\xx)}_{p^*} \leq (1+\sigma) \gamma \norm{\xx-\yy}_p^{p-1}.
    \end{equation*}
    Additionally, we say $\Oracle$ satisfies an \emph{$(s,\mu)$ movement bound} if 
    \begin{equation*}
        \norm{\xx -  \yy}_p  \geq \begin{cases}
            (\gamma/\mu^s)^{1/(s-p)} & \text{ if $s<\infty$}\\
            1/\mu & \text{ if $s = \infty$}.
        \end{cases}
    \end{equation*}
\end{definition}
In Appendix~\ref{app:LSF}, we show that $\sigmaGMS$ oracles generalize $\ell_p^s$ oracles as well as the Monteiro-Svaiter (MS) oracles from~\cite{monteiro2013accelerated,carmon2022optimal}. In this section we provide the following results.
\begin{theorem}
\label{lem:finite-s-lemma}
       Algorithm~\ref{alg:LSFProx}, given $s > p \geq 2$, $\epsilon > 0$, $\xx_0 \in \R^d$, $R = O(\norm{x_0-x^*}_p)$, and a $\sigmaGMS$ oracle satisfying an $(s,\mu)$ movement bound, returns $\xxtil \in \R^d$ such that $\ff(\xxtil)- \ff(\xx^{\star})\leq \epsilon$ after $$ O\left((\mu^s \|\xx_0-\xx^{\star}\|_p^s / \epsilon)^{\frac{1}{s(1+\nu)}}\right)$$ iterations. Each iteration can be implemented in $O(d)$ time plus time of the $\sigmaGMS$ oracle.
\end{theorem}

\begin{theorem}\label{lem:Inf-s-lemma}
    Algorithm~\ref{alg:LSFProx}, given $p \geq 2$, $\epsilon > 0$, $x_0 \in \R^d$, $R = O(\norm{x_0-x^*}_p)$, and a $\sigmaGMS$ oracle satisfying an $(\infty,\mu)$ movement bound, returns $\xx_T$ such that
$ \ff(\xx_T) - \ff(\xx^{\star} ) \leq \epsilon$ after 
    $$O\left(\left(\mu p \|\xx_0-\xx^{\star}\|_p\right)^{\frac{p}{p+1}} \log\frac{\|\xx_0-\xx^{\star}\|_p}{\epsilon}\right)$$
    iterations.     Each iteration can be implemented in $O(d)$ time plus time of the $\sigmaGMS$ oracle.
\end{theorem}

\newcommand{\lbar}{\bar{\lambda}_{t+1}}
\begin{algorithm}
\caption{General Proximal Point Algorithm (Line-Search Free)}
\label{alg:LSFProx}
  \textbf{Initialize:} $s > p \geq 2$, $a_0 = 1$, $A_0 = 0$, $\bar{\lambda}_1 = 1$,
  $\zz_0 = \yy_0 =\xx_0 \in \R^d$, $T \geq 1$, $T_{\mathrm{ref}} = 1$, $R > 0$, $\mu > 0$, $\sigma \in [0,1)$\\
\For{$t = 0,\cdots, T-1$}{
\If{$t\geq 1$}{
\lIf{$A_t \geq 2A_{T_{\mathrm{ref}}}$}{
$T_{\mathrm{ref}} = t$
}
$\lbar =\begin{cases}
    A_{T_{\mathrm{ref}}}^{-\frac{(s-p)(p+1)}{ps-p+s}} R^{\frac{p(s-p)}{ps-p+s}} \mu^{\frac{sp^2}{s-p+ps}} & \text{ if $s<\infty$}\\[0.2cm]
      \frac{\mu^{\frac{p^2}{p+1}}R^{\frac{p}{p+1}}}{A_{T_{\mathrm{ref}}} } & \text{ if $s = \infty$}
  \end{cases}$
}
  $\lbar (\frac{3+3\sigma}{1-\sigma} a_{t+1})^p = A_{t+1}'^{p-1}$, $A_{t+1}' = A_t + a_{t+1}$
 $\yy_t \leftarrow \frac{A_t}{A_{t+1}}\xx_t + \frac{a_{t+1}}{A_{t+1}}\zz_t$\label{algline:Iterate}\\
 $(\xx_{t+1}', \lambda_{t+1}) \gets \sigmaGMS(\yy_t)$ \\
 $\beta_{t+1} \gets \min \{ 1, \frac{ \lbar }{\lambda_{t+1}} \}$\\
 $\zz_{t+1} \gets \arg\min_{\zz\in \rea^d} \sum_{i=1}^{t+1} a_i \beta_i \nabla \ff(\xx_i')^{\top}(\zz-\yy_i) + \omega_p(\zz,\yy_0)$\\
 $\xx_{t+1} \gets \frac{ (1- \beta_{t+1}) A_t x_t + \beta_{t+1} A_{t+1}' x_{t+1}'}{A_t + \beta_{t+1} a_{t+1}}$\\
 $A_{t+1} \gets A_t + \beta_{t+1} a_{t+1}$\\
 }
 \Return $\xx_{T}$
\end{algorithm}
The proofs of the above results follow from the following potential analysis, which differs from those in Section~\ref{sec:Prox} in that it now accounts for the progress differently depending on whether $\lbar < \lambda_{t+1}$ or $\lbar\geq \lambda_{t+1}$.
\begin{lemma}\label{lem:Potential2}
   The iterates of \Cref{alg:LSFProx} satisfy
\begin{align*}A_{t+1}\left(\ff(\xx_{t+1})-\ff(\xx^{\star})\right) + \omega_p(\xx^{\star},\zz_{t+1}) \leq & A_{t}\left(\ff(\xx_{t})-\ff(\xx^{\star})\right) + \omega_p(\xx^{\star},\zz_{t}) \\
&- \mathbbm{1}_{\{\lbar<\lambda_{t+1}\}}\frac {A_{t+1}' (1-\sigma) \lbar}{3} \|\xx'_{t+1}-\yy_t\|_p^p,
\end{align*}
and, $A_{t+1}^{1/p} \geq A_{t}^{1/p} + \mathbbm{1}_{\{\lbar\geq \lambda_{t+1}\}}\frac{1}{\frac{6+6\sigma}{1-\sigma} \lbar^{1/p}}$.
\end{lemma}
We defer the proofs of both the potential analysis and our main results in this section to Appendix~\ref{app:LSF}.

\section{Applications}\label{sec:App}

In this section, we cover different applications of our acceleration framework from Section~\ref{sec:NoLS}. In particular, we apply our framework to $\ell_s$-norm regression, where $s\geq p\geq 2$, higher-order smooth optimization, and implement an $\ell_p$-ball oracle.

\paragraph{$\ell_s$-Norm Regression.}\label{sec:regression}

We give an algorithm for $\ell_s$-regression, where $s\geq p\geq 2$ using an oracle that solves problems of the form,
\begin{equation}\label{eq:pNormOracle}
\min_{\xx\in \mathbb{R}^d} \gg^{\top}\xx + \|\RR\xx\|_2^2 + \|\WW\xx\|_p^p,
\end{equation}
for any vector $\gg\in \mathbb{R}^d$ and matrices $\RR\in \rea^{n\times d}$ and $\WW \in \rea^{n\times d}$, where $n \geq d$. 
We will prove Theorem~\ref{thm:regression-main} which we restate here.
\SNormReg*

We present the proof in Appendix~\ref{app:app}. The proof follows from an application of Algorithm~\ref{alg:LSFProx} for the $\ell_p^s(\lambda)$-proximal oracle (Theorem~\ref{thm:ConceptualProx-main}), which reduces solving the $\ell_s$-regression problem to $\approx n^{\frac{\nu}{1+\nu}}$ calls to an $\ell_p^s$-proximal oracle. We further prove that every such call to the $\ell_p^s$-proximal oracle can be solved using $\widetilde{O}(1)$ smoothed $\ell_p$-norm problems, i.e., Eq.~\eqref{eq:pNormOracle} (Lemma~\ref{lem:proxSubProb}).

\paragraph{High-Order Smooth Optimization in General Norms.}\label{sec:ho-smooth}

As an extension of our results, we consider a natural relaxation of the generalized proximal oracle based on regularized high-order Taylor approximations~\cite{bullins2020highly, gasnikov2019near, monteiro2013accelerated, nesterov2006cubic, nesterov2008accelerating}. 
Namely, letting
\[
\mathcal{T}_f^{q}(\yy;\xx) \defeq \ff(\xx) + \sum\limits_{i=1}^{q} \frac{1}{i!}\nabla^{i} \ff(\xx)[\yy-\xx]^{i-1}
\]
denote, for $q \geq 1$, the $q^{th}$-order Taylor expansion of $\ff$ at $\xx$, we define the following notion of high-order smoothness, which naturally generalizes those in previous works~\cite{nesterov2008accelerating,gasnikov2019near}.
\begin{definition}\label{def:HO-smooth}
    We say $\ff$ is \emph{$q^{th}$-order $L_{q}$ smooth with respect to $\norm{\cdot}_p$} if, for all $\xx, \yy \in \rea^d$,
\begin{equation}
    \norm{\nabla \ff(\yy) - \nabla_{\yy} \mathcal{T}_f^{q}(\yy;\xx)}_{p^*} \leq \frac{L_{q}}{q!}\norm{\yy - \xx}_p^{q}.
\end{equation}
\end{definition}
We now consider the problem of minimizing a convex function $\ff$ given what we call a \emph{$(p,s)$-Taylor oracle}, which, for a centering point $\cc$, returns the solution of the regularized problem 
\[
\min_{x\in \rea^d} \Taylor(\xx;\cc) + \frac{L_{s-1}}{(s-1)!} \|\xx-\cc\|_p^{s}.
\]
In contrast to the $\ell_p^s(\lambda)$-proximal oracle, here we have replaced $\ff(\xx)$ with its $(s-1)^{th}$-order counterpart $\Taylor(\xx; \cc)$, and the smoothness parameter now fulfills the role of $\lambda$ (which was previously a parameter of the oracle).
As a natural consequence of our results from Section~\ref{sec:NoLS}, combined with the observation that that our Taylor oracle implements an $O(1)$-$\GMS$ oracle, we provide the following convergence guarantees for Taylor oracle-based algorithms.

\thmHO*

We would note Theorem~\ref{thm:HO} obtains the known optimal rates in the case of $p = 2$ and $s > 2$ \cite{arjevani2019oracle, gasnikov2019near}, and thus our guarantees allow us to generalize these previous results to all $s > p \geq 2$.

\paragraph{Implementing the $\ell_p$ Ball Oracle Under Hessian Stability.}
As in the case of the Euclidean ball oracle~\cite{carmon2020acceleration, karimireddy2018global}, we rely on the following notion of \emph{Hessian stability} for implementing our more general $\ell_p$ ball oracles. 
\begin{definition}[\cite{carmon2020acceleration}, Definition 7]
    A twice-differentiable function $f : \R^d \rightarrow \R$ is \emph{$(r,\xi)$-Hessian stable with respect to $\norm{\cdot}$}, for $r, \xi > 0$, if for all $x, y \in \R^d$ such that $\norm{x-y} \leq r$, we have $\xi^{-1} \nabla^2 f(y) \leq \nabla^2 f(x) \leq \xi \nabla^2 f(y)$.
\end{definition}

As its name suggests, this type of stability provides control over \emph{local} changes in the Hessian, and objectives such logistic regression and the softmax loss exhibit favorable Hessian stability with respect to $\norm{\cdot}_2$ and $\norm{\cdot}_\infty$, respectively~\cite{bach2010self,carmon2020acceleration,karimireddy2018global}. We may further observe that such control entails beneficial (local) conditioning (i.e., strong convexity and smoothness) of $f$ \emph{with respect to $\norm{\cdot}_{\nabla^2 f(y)}$} (where we let $\norm{u}_{\nabla^2 f(y)}^2 = u^\top\nabla^2 f(y) u$), and it can be shown that this condition holds when $f$ is \emph{$\alpha$-quasi-self-concordant with respect to $\norm{\cdot}$}, meaning that, for all $x, u, v \in \R^d$, $|\nabla^3 f(x)[v,u,u]| \leq \alpha\norm{v}\norm{u}_{\nabla^2 f(x)}^2$. It is natural, then, that we may run an appropriate descent method (in terms of $\norm{\cdot}_{\nabla^2 f(y)}$, e.g., \cite{carmon2020acceleration}, Theorem 9), constrained to the region of stability.

We additionally rely on a relaxation of the $\ell_p$ ball oracle (Definition~\ref{def:ball}), which allows for a certain amount of approximation in a manner similar to that in \cite{carmon2020acceleration}.

\begin{definition}\label{def:approxBall} Given $p \geq 2$, $\delta, r > 0$, an $\ell_p^{\infty}(\delta,r)$-proximal oracle (also referred to as an \emph{approximate ball-constrained oracle}) for a function $\ff:\R^d\rightarrow \R$ is an oracle that, when queried at a centering point $\cc\in \R^d$, returns $\tilde{x} \in \R^d$ such that $\norm{\tilde{x}-x_{c,r}^*}_p \leq \delta$, where $x_{c,r}^*$ is solution of the problem 
\[
\min_{x\in \rea^d: \|\xx-\cc\|_p\leq r} \ff(\xx).
\]
\end{definition}
The following theorem establishes that, under an appropriate choice of $\delta$, such an oracle can implement a $\sigma$-$\textrm{GMS}_p$ oracle, and its proof can be found in Appendix~\ref{sec:proofApproxImpGMS}.

\begin{theorem}\label{thm:approxImpGMS}
    Let $f$ be $L$-smooth w.r.t. $\norm{\cdot}_p$, and let $\tilde{x} \in \R^d$ be given by an $\ell_p^{\infty}(\delta,r)$-proximal oracle for $f$ queried at $c \in \R^d$, for $p \geq 2$, $r, \tilde{\epsilon}, \sigma > 0$, $\min\{\frac{r\sigma}{4},\frac{\tilde{\epsilon}\sigma}{4rL}\} \geq \delta > 0$. Then, if $\norm{\nabla f(\tilde{x})} \geq \tilde{\epsilon}$, an oracle that returns $(\tilde{x}, \tilde{\gamma})$ for $\tilde{\gamma} = \frac{\norm{\nabla f(\tilde{x})}}{\norm{\tilde{x} - y}_p^{p-1}}$ is a $\sigma$-$\textrm{GMS}_p$ oracle, and furthermore it satisfies an $(\infty, (r-\delta)^{-1})$ movement bound. 
\end{theorem}

\begin{algorithm}[h]
\caption{Constrained Accelerated Newton Descent}
\label{alg:genGD}
  \textbf{Initialize:}
  $a_0=1$, $A_0=0$, $\xx_0 = \yy_0 = \zz_0 = \xx_{\textrm{init}} \in \R^d$, $r > 0$, $T \geq 1$\\
\For{$t = 0,\cdots, T-1$}{
$5a_{t+1}^2 = A_{t+1}$, $A_{t+1} = A_t + a_{t+1}$\\\
$y_t \leftarrow \frac{A_t}{A_{t+1}}x_t + \frac{a_{t+1}}{A_{t+1}}z_t$\\
$\xx_{t+1} \leftarrow {\arg\min}_{x : \norm{x-c} \leq r} \{\inner{\nabla f(y_t)}{x-y_t} + \xi\norm{x-y_t}_{\nabla^2 f(\cc)}^2\}$\\
$z_{t+1} \leftarrow {\arg\min}_{z} \sum_{i \in [t+1]} a_i\nabla f(x_i)^\top(z-y_i) + \norm{z-y_0}_{\nabla^2 f(c)}^2$
 }
\Return $x_T$
\end{algorithm}

\begin{algorithm}[h]
\caption{Constrained Accelerated Newton Descent + Restarting}
\label{alg:restart}
  \textbf{Initialize:}
  $\delta > 0$, $\xx_{0} \in \R^d$, $r > 0$, $K = O(\log(\frac{\xi dLr}{\mu\delta}))$\\
\For{$k = 0,\cdots, K-1$}{
$x_{k+1} = \textrm{Constrained Accelerated Newton Descent}(\xx_{\textrm{init}} = \xx_k, r, T = O(\xi)$)
 }
\Return $x_K$
\end{algorithm}

It remains to show how we may implement an $\ell_p^{\infty}(\delta,r)$-proximal oracle for Hessian stable functions, which the following theorem provides.

\begin{theorem}\label{thm:implementApproxOracle}
Let $f$ be $L$-smooth and $\mu$-strongly convex w.r.t. $\norm{\cdot}_2$, and $(r,\xi)$-Hessian stable w.r.t. $\norm{\cdot}_p$. Then, there is an algorithm (Algorithm~\ref{alg:restart}) that implements an $\ell_p^{\infty}(\delta,r)$-proximal oracle, using $O(\xi\log(\frac{\xi dLr}{\mu\delta}))$ queries to a gradient oracle and one query to a Hessian oracle of $f$.
\end{theorem}

The proof (found in Appendix~\ref{sec:proofImpApproxOracle}) closely follows that of Theorem~\ref{thm:ConceptualProx}, along with a standard restarting argument (e.g.,~\cite{roulet2017sharpness}) to attain linear convergence rates when paired with the strong convexity that comes from Hessian stability.

\section{Lower Bounds}\label{sec:LB}
In this section, we complement our upper bound results with matching lower bounds, for both $s < \infty$ and $s = \infty$ cases, for algorithms that are \emph{zero-respecting}~\cite{carmon2020lower}. This condition, which entails, roughly speaking, that the support of any iterate is restricted by the span of all preceding iterates, plays an important role in establishing lower bounds for a variety of optimization settings~\cite{arjevani2023lower, carmon2021lower}.

\paragraph{When \texorpdfstring{$s<\infty$}{TEXT}: Proximal Point Oracles.}
We begin by proving lower bounds for zero-respecting proximal point algorithms when minimizing a function $\ff$ using an $\ell_p^s(\lambda)$-proximal oracle, where we let $\calO_{\ff, \lambda,p,s}(c)$ denote the result upon querying the oracle at $c \in \R^d$.

We first introduce an appropriate notion of \emph{zero-respecting algorithms} for our setting. Namely, letting $\supp{\vv} := \setof{i \in [d] : \vv_i \neq 0}$ for $\vv \in \rea^d$, we say a sequence $\xx^{(0)}, \xx^{(1)}, \dots$ is \emph{$\ell_p^s(\lambda)$-proximal zero-respecting with respect to $\ff$} if
\[
\supp{\xx^{(t)}} \subseteq \bigcup_{t' < t} \supp{\calO_{\ff, \lambda,p,s}(\xx^{(t')})}\ \textrm{for each}\ t \in \nat.
\]

We may then naturally define an \emph{$\ell_p^s(\lambda)$-proximal zero-respecting} algorithm.

\begin{definition}[$\ell_p^s(\lambda)$-Proximal Zero-Respecting Algorithm]\label{def:zero-resp-alg}
    We say an algorithm $\calA$ is \emph{$\ell_p^s(\lambda)$-proximal zero-respecting} if, for any $f: \rea^d \rightarrow \rea$, the iterate sequence $\calA[f]$ is zero-respecting with respect to $f$.
\end{definition}

We further consider the following proximal-type notion of a \emph{zero-chain}.
\begin{definition}[$\ell_p^s(\lambda)$-Proximal Zero-Chain]
    A function $\ff: \rea^d \rightarrow \rea$ is an \emph{$\ell_p^s(\lambda)$-proximal zero-chain} if, for every $\xx \in \rea^d$,
        $\supp{\xx} \subseteq \setof{1,\dots,i-1}$ implies $\supp{\calO_{\ff, \lambda,p,s}(\xx)} \subseteq \setof{1,\dots,i}$.
\end{definition}
We would note that, while previous lower bounds for smooth convex optimization~\cite{arjevani2019oracle, nemirovskij1983problem} hold for algorithms that can generate approximate proximal point oracle updates (up to polylogarithmic error), here we provide lower bounds for when we have an \emph{exact} proximal point oracle. 
We now state our main lower bound theorem for this section, where we recall that $x^* := \arg\min_{\xx \in \R^d} g_k(\xx)$.

\LowerBound*

While we provide the full proof of Theorem~\ref{thm:LowerBound-main} in Appendix~\ref{app:LBsTheorem}, the key construction is based on a scaled instance of Nemirovski's function~\cite{nemirovskij1983problem}, whose parameters depend on properties (i.e., $p, s, \lambda$) defining the proximal oracle.

\begin{definition}[Scaled Nemirovski instance]\label{def:Nemirovski}
Let $k, d \in \nat$ be such that $1 \leq k\leq d$, and let $R > 0$. We define
\[
\ff_k(\xx) = \beta_k\max_{1\leq i\leq k}\{ \xx_i - i\cdot \alpha_k\},
\]
where 
\[
\beta_k := \left(\frac{s-1}{s}\right)^{2(s-1)}\frac{s\lambda R^{s-1}}{(k+1)^{\frac{(p+1)(s-1)}{p}}},\ \alpha_k := \frac{s}{s-1}\left(\frac{\beta_k}{s\lambda}\right)^{\frac{1}{s-1}} = \frac{(s-1)R}{s(k+1)^{\frac{p+1}{p}}}.
\]
\end{definition}

\noindent Following the approach as described in previous works~\cite{carmon2020acceleration, diakonikolas2019lower}, we may handle unconstrained domains by considering the following problem:
\begin{equation}\label{eq:HardProb}
    \min_{\xx \in R^d} \gg_k(\xx) := \max\left\{\ff_k(\xx), \beta_k(\norm{\xx}_p - 2R) - \alpha_k\right\}.
\end{equation}

The following lemma, which is proved in Appendix~\ref{app:LB}, shows that the iterates produced by our $\ell_p^s(\lambda)$-proximal oracle satisfy the required condition on their support.  
\begin{restatable}{lemma}{LemSpan}\label{lem:Span}
Let $k, d \in \nat$ be such that $1 \leq k\leq d$, and let $\gg_k$ be as in \eqref{eq:HardProb}. Then, $\gg_k$ is an $\ell_p^s(\lambda)$-proximal zero-chain.
\end{restatable}

\paragraph{When \texorpdfstring{$s = \infty$}{TEXT}: Ball Oracle.}
For this case, we consider the $\ell_p$ ball oracle, i.e.,
\[
\calO_{\ff,r,p,\infty}(\cc) = {\arg\min}_{\|\xx-\cc\|_p\leq r} \ff(\xx).
\]

Similar to the previous section, we may define analogous notions of \textit{$\ell_p^{\infty}(r)$-proximal zero-respecting algorithms} and an \textit{$\ell_p^{\infty}(r)$-proximal zero-chain}. We say a sequence $\xx^{(0)}, \xx^{(1)}, \dots$ is \emph{$\ell_p^{\infty}(r)$-proximal zero-respecting with respect to $\ff$} if
\[
\supp{\xx^{(t)}} \subseteq \bigcup_{t' < t} \supp{\calO_{\ff, r,p,\infty}(\xx^{(t')})}\ \textrm{for each}\ t \in \nat,
\]
We may then similarly define a \emph{$\ell_p^{\infty}(r)$-proximal zero-respecting} algorithm.

\begin{definition}[$\ell_p^{\infty}(r)$-Proximal Zero-Respecting Algorithm]\label{def:infProxAlg}
    We say an algorithm $\calA$ is \emph{$\ell_p^{\infty}(r)$-proximal zero-respecting} if, for any $f: \rea^d \rightarrow \rea$, the iterate sequence $\calA[f]$ is $\ell_p^{\infty}(r)$-proximal zero-respecting with respect to $f$.
\end{definition}

We further consider the following ball oracle-type notion of a $\ell_p^{\infty}(r)$-\emph{zero-chain}.
\begin{definition}[$\ell_p^{\infty}(r)$-Proximal Zero-Chain]
    A function $\ff: \rea^d \rightarrow \rea$ is an \emph{$\ell_p^{\infty}(r)$-proximal zero-chain} if, for every $\xx \in \rea^d$,
        $\supp{\xx} \subseteq \setof{1,\dots,i-1}$ implies $\supp{\calO_{\ff, r,p,\infty}(\xx)} \subseteq \setof{1,\dots,i}$.
\end{definition}

We now state the main result (Theorem~\ref{thm:LowerBoundInf-main}) of this section, whose proof can be found in Appendix~\ref{app:LBinfThm}.
\LowerBoundInf*

\section*{Acknowledgments}
DA is supported by Dr. Max Rössler, the Walter Haefner Foundation and the ETH Zürich Foundation. AS  was supported in part by a Microsoft Research Faculty Fellowship, NSF CAREER Grant CCF-1844855, NSF Grant CCF-1955039, and a PayPal research award. Part of this work was conducted while the authors were visiting the Simons Institute for the Theory of Computing.

\printbibliography

\newpage
\appendix

\section{Preliminaries}

\begin{definition}[Bregman Divergence]
    For a strictly convex, continuously differentiable  $\dd:D \rightarrow \rea_{\geq 0}$, we define its Bregman divergence between $x\in D$ and $y\in D$ as
    \[
    \omega(\xx,\yy) := \dd(\xx) - \dd(\yy) - \langle \nabla\dd(\yy), \xx-\yy\rangle.
\]
\end{definition}
    Note that $\omega_p$ is the Bregman divergence of the function $\dd(\xx) = \|\xx-\xx_0\|_p^p$ for an initial point $\xx_0$. 

\begin{lemma}[Three-Point Property]\label{lem:3Point} For any $x,y,z \in D$,
\[
    \omega(\xx,\yy) = \omega(\xx,\zz) + \omega(\zz,\yy) - (\xx-\zz)^{\top}(\nabla \dd(\yy)-\nabla\dd(\zz)).
\]
\end{lemma}

We also use the following result, which can be found in, e.g.,~\cite{zǎlinescu1983uniformly}.

\begin{lemma}[Lemma 3.1~\cite{zǎlinescu1983uniformly}]\label{lem:PrecAKPS}
    For any $\xx, \yy$ and $p \geq 2$, $\omega_p(\yy,\xx)\geq \frac{1}{2^{p+1}}\|\yy-\xx\|_p^p.$
    \end{lemma}
We use the following fact about the gradient and Hessian of the $\ell_p$-norm function repeatedly.
\begin{fact}\label{fact:grad}
    The gradient and Hessian of the function $\|Ax\|_p^p$, for any $p\geq 2$, and matrix $A\in \rea^{n\times d}$ and vector $x\in \rea^d$ is,
    \[
    \nabla_{\xx} \|Ax\|_p^p = p \cdot  \diag{|Ax|^{p-2}}Ax, \quad \nabla^2_{xx}\|Ax\|_p^p = p(p-1) A^{\top}\diag{|Ax|^{p-2}}A.
    \]
\end{fact}
\section{Proofs for Conceptual Method}\label{app:concept}

\subsection{Proof of Lemma~\ref{lem:Potential}}
\begin{proof}
    We start by considering the left-hand side of the inequality
    \begin{align*}
        A_{t+1}&\left(\ff(\xx_{t+1})-\ff(\xx^{\star})\right) - A_{t}\left(\ff(\xx_{t})-\ff(\xx^{\star})\right)\\
       & = a_{t+1}\left(\ff(\xx_{t+1})-\ff(\xx^{\star})\right) + A_t \left(\ff(\xx_{t+1})-\ff(\xx_t)\right)\\
       & \leq  a_{t+1}\langle \nabla \ff(\xx_{t+1}),\xx_{t+1}-\xx^{\star}\rangle + A_t\langle \nabla \ff(\xx_{t+1}),\xx_{t+1}-\xx_t\rangle &&\text{(convexity of $\ff$)}\\
      & =  \langle \nabla \ff(\xx_{t+1}),A_{t+1}\xx_{t+1} - A_t\xx_t - a_{t+1}\xx^{\star}\rangle\\
      & =  \langle \nabla \ff(\xx_{t+1}),A_{t+1}\xx_{t+1} +a_{t+1}\zz_t-A_{t+1}\yy_t  - a_{t+1}\xx^{\star}\rangle &&\text{(Since $A_{t+1} y_t = A_t x_t + a_{t+1} z_t$)}\\
      & =  A_{t+1}\underbrace{\langle\nabla \ff(\xx_{t+1}),\xx_{t+1} -\yy_t  \rangle}_{\text{Term 1}} + a_{t+1}\underbrace{\langle \nabla \ff(\xx_{t+1}),\zz_t  - \xx^{\star}\rangle}_{\text{Term 2}}.
    \end{align*}
    We first give a bound on Term 2. Using the KKT conditions for $z_t$, i.e., differentiating the equation defining $z_t$ in the algorithm w.r.t $z$, we have for all $t \geq 0$,
    \[
    \sum_{i \in [t]} a_i \nabla \ff(\xx_i) = - \nabla_{\zz}\|\zz-\yy_0\|_p^p\Big\vert_{\zz=\zz_t}
    = - \nabla\omega_p(\zz_t,\yy_0)\,
    \enspace\text{for all}\enspace t,
    \]
    where we used that $\omega_p(\zz_t,\yy_0) = \norm{\zz_t - \yy_0}_p^p$ since $\xx_0=\yy_0$. Therefore, 
    \[
    a_{t+1} \nabla\ff(\xx_{t+1}) = \nabla \omega_p(\zz_t,\yy_0) - \nabla \omega_p(\zz_{t+1},\yy_0).
    \]
    Replacing $a_{t+1}\nabla\ff(\xx_{t+1})$ with the above in Term 2, we get
    \begin{align*}
        & a_{t+1}\langle \nabla \ff(\xx_{t+1}),\zz_t  - \xx^{\star}\rangle
        =   a_{t+1}\langle \nabla \ff(\xx_{t+1}),\zz_{t+1}  - \xx^{\star}\rangle +  a_{t+1}\langle \nabla \ff(\xx_{t+1}),\zz_t  - \zz_{t+1}\rangle\\
        = & \langle \nabla \omega_p(\zz_t,\yy_0) - \nabla \omega_p(\zz_{t+1},\yy_0), \zz_{t+1} - \xx^{\star}\rangle + a_{t+1}\langle \nabla \ff(\xx_{t+1}),\zz_t  - \zz_{t+1}\rangle.
        \end{align*}
        Applying the Bregman three-point property (Lemma~\ref{lem:3Point}) that yields that
        \[
         a_{t+1}\langle \nabla \ff(\xx_{t+1}),\zz_t  - \xx^{\star}\rangle  = \omega_p(\xx^{\star},\zz_{t}) - \omega_p(\xx^{\star},\zz_{t+1}) - \omega_p(\zz_{t},\zz_{t+1}) + a_{t+1}\langle \nabla \ff(\xx_{t+1}),\zz_t  - \zz_{t+1}\rangle\,.
        \]
      We now apply Young's inequality to get
      \begin{multline*}
       a_{t+1}\langle \nabla \ff(\xx_{t+1}),\zz_t  - \xx^{\star}\rangle\leq \omega_p(\xx^{\star},\zz_{t}) - \omega_p(\xx^{\star},\zz_{t+1}) - \omega_p(\zz_{t},\zz_{t+1})\\  + \frac{p-1}{p}\|2 a_{t+1}\nabla \ff(\xx_{t+1})\|_{p^*}^{p^*} + \frac{1}{p2^p}\|\zz_t-\zz_{t+1}\|_p^p.
      \end{multline*}
    Since $\omega_p(\xx,\yy) \geq \frac{1}{2^{p+1}}\|\xx-\yy\|_p^p$ (Lemma~\ref{lem:PrecAKPS}), we obtain the following bound on Term 2,
    \begin{equation}\label{eq:Term2prox}
            a_{t+1}\langle \nabla \ff(\xx_{t+1}),\zz_t  - \xx^{\star}\rangle \leq \omega_p(\xx^{\star},\zz_{t}) - \omega_p(\xx^{\star},\zz_{t+1})  + 4a_{t+1}^{p^*}\frac{p-1}{p}\|\nabla \ff(\xx_{t+1})\|_{p^*}^{p^*}.
    \end{equation}
    
    We next bound Term 1. First, we observe by the optimality guarantees of $\xx_{t+1}$ that
    \[
    \nabla \ff(\xx_{t+1}) =  -\lambda_t p \cdot  \diag{|\xx_{t+1}-\yy_t|}^{p-2}(\xx_{t+1}-\yy_t),
    \]
    where the above derivative is computed using Fact~\ref{fact:grad}.
    Therefore,
    \begin{equation}\label{eq:Term1prox}
        \langle\nabla \ff(\xx_{t+1}),\xx_{t+1} -\yy_t  \rangle
       =-\lambda_t\norm{\xx_{t+1} -\yy_t}_p^p.
        \end{equation}
Combining the bounds on Term 1~\eqref{eq:Term2prox} and Term 2~\eqref{eq:Term1prox} yield 
      \begin{align*}
         &  A_{t+1}\left(\ff(\xx_{t+1})-\ff(\xx^{\star})\right) - A_{t}\left(\ff(\xx_{t})-\ff(\xx^{\star})\right)\\
         \leq & - A_{t+1} \langle \nabla\ff(\xx_{t+1}), \xx_{t+1} -\yy_t \rangle + \omega_p(\xx^{\star},\zz_{t}) - \omega_p(\xx^{\star},\zz_{t+1})  + 4a_{t+1}^{p^*}\frac{p-1}{p}\|\nabla \ff(\xx_{t+1})\|_{p^*}^{p^*}\\
         \leq & - A_{t+1}\lambda_t \|\xx_{t+1}-\yy_t\|_p^p + \omega_p(\xx^{\star},\zz_{t}) - \omega_p(\xx^{\star},\zz_{t+1})  + 4a_{t+1}^{p^*}\lambda_t^{p^*}\cdot \frac{p-1}{p} \|\xx_{t+1}-\yy_t\|_{p}^{p}
      \end{align*}
      Now, since $A_{t+1}^{p-1} =  5^{p-1}\lambda_t a_{t+1}^p$, the above becomes
      \begin{equation*}
      A_{t+1}\left(\ff(\xx_{t+1})-\ff(\xx^{\star})\right)  + \omega_p(\xx^{\star},\zz_{t+1})
      \leq A_{t}\left(\ff(\xx_{t})-\ff(\xx^{\star})\right) + \omega_p(\xx^{\star},\zz_{t}) - A_{t+1}\lambda_t \|\xx_{t+1}-\yy_t\|_{p}^{p}.
      \end{equation*}
\end{proof}

\subsection{Proof of Theorem~\ref{thm:ConceptualProx}}
We need to prove the existence of $\lambda_t$ $(t \geq 0)$ satisfying the conditions of Algorithm~\ref{alg:ProxQ}, using a similar continuity argument as in Lemma 3.2 of~\cite{bubeck2019near}.
\begin{lemma}\label{lem:Continuity}
    Let $A\geq 0, \xx,\yy\in \R^d$ such that $\ff(\xx)\neq \ff(\xx^{\star})$. Also let $\zz(\xx)$ be a continuous function of $\xx$. Define the following functions:
    \begin{itemize}
        \item $a(\lambda)$ is the solution of the equation $\lambda^{\frac{1}{p-1}}a(\lambda)^{p^*} -a(\lambda)-A=0$.
        \item $\yy(\lambda) = \frac{a(\lambda)}{A+a(\lambda)}z(\xx) + \frac{A}{A+a(\lambda)}\xx$
        \item $\xx(\lambda) = \arg\min_{\ww} \ff(\ww) + \lambda\|\ww-\yy(\lambda)\|_p^p$
        \item $\gg(\lambda) = \frac{\|\xx(\lambda)-\yy(\lambda)\|_p^{s-p}}{\lambda}$
    \end{itemize}
    Then, for every $z\in \mathbb{R}_+$ there exists $w\in \mathbb{R}_+$ such that $\gg(w) = z$. 
\end{lemma}
\begin{proof}
We first claim that $\gg(\lambda)$ is a continuous function of $\lambda$. It is easy to see that $\yy(\lambda)$ is continuous. The continuity of $\xx(\lambda)$ follows from the fact that since $\ff$ is convex, $\ff(\ww) + \lambda\|\ww-\cc\|_p^p$ is strictly convex and there is a unique minimizer for every $\lambda$. Now, $\xx\neq \yy$ since $\ff(\xx) \neq \ff(\xx^{\star})$. Therefore, $\gg(0) = \infty$ and $\gg(\infty) = 0$, which concludes the proof. 
\end{proof}
The following lemma provides a means of lower bounding $A_t$ for all $t \geq 1$. 
\begin{lemma}\label{lem:ARelations}
For all $t \geq 1$, Algorithm~\ref{alg:ProxQ} guarantees that
  \[
  A_{t+1}^{1/p}\geq A_t^{1/p} + \frac{1}{5p\lambda_t^{1/p}}
  	\text{\ \  and, consequently,\ \  }
  A_t^{1/p} \geq A_0^{1/p} + \sum_{t \in [T]} \frac{1}{5 p \lambda_t^{1/p}}\,.
  \] 
\end{lemma}
\begin{proof}
    We know from Algorithm~\ref{alg:ProxQ} that $A_{t+1} = A_t + a_{t+1}$ and $a_{t+1}^p = \frac{A_{t+1}^{p-1}}{5^{p-1}\lambda_t}$. Therefore,
    \[
    A_t = A_{t+1} - \left(\frac{A_{t+1}^{p-1}}{5^{p-1}\lambda_t}\right)^{1/p} = A_{t+1} \left( 1 - \frac{1}{(5^{p-1}\lambda_t A_{t+1})^{1/p}}\right).
    \]
    We will use that for any $x>-1$ and $r>1$, $(1+x)^r \geq 1+rx.$ Observe that $\frac{1}{(5^{p-1}\lambda_t A_{t+1})^{1/p}} = \frac{a_{t+1}}{A_{t+1}} <1 $. Applying this inequality for $r = p$ and $x = - \frac{1}{(5^{p-1}\lambda_t A_{t+1})^{1/p}}$, we get
    \[
    A_t  \leq A_{t+1} \left( 1- \frac{1}{(5^{p-1}p\lambda_t A_{t+1})^{1/p}}\right)^p,
    \]
    which on rearranging gives us
    \begin{equation*}
    A_t^{1/p} +\frac{1}{5p \lambda_t^{1/p}} \leq A_{t+1}^{1/p}.
    \end{equation*}
\end{proof}
Before we prove our main result, we state a result from~\cite{carmon2022optimal} which is useful in our final proof.
\begin{lemma}[Lemma 3,~\cite{carmon2022optimal}]\label{lem:LBSeqA}
Let $B_1, \cdots, B_k \in \mathbb{R}_{>0}$, $r_1, \cdots, r_k\in \mathbb{R}_{>0}$  satisfy $B_i^m
 \geq \beta \sum_{j\in[i]} B_j r_j$ for some $m > 1$, $\beta>0$, and all $i\in [k]$. Then $B_i \geq (\frac{m- 1}{m} \beta \cdot \sum_{j\in [i]} r_j)^{\frac{1}{m-1}}$ for all $i \in [k]$.
\end{lemma}
We are now ready to prove the main result of Section~\ref{sec:Prox}. 
\ThmConceptProx*
\begin{proof}[Proof of \Cref{thm:ConceptualProx}]
Let $r_t \defeq  \|\xx_{t+1}- \yy_t\|_p$. Note that $\lambda_t = \lambda r_t^{s-p}$ and therefore $\lambda_t \|\xx_{t+1}- \yy_t\|_p^p = \lambda_t^{\frac{s}{s- p}}\lambda^{-\frac{p}{s-p}}$. Since we are using an $\ell_p^s(\lambda)$-proximal oracle, the conditions of \Cref{lem:Potential} are true. The lemma then implies that for all $T \geq 1$
\begin{equation}\label{eq:PotHelp}
A_{T}\left(\ff(\xx_{T})-\ff(\xx^{\star})\right) + \omega_p(\xx^{\star},\zz_{T}) \leq A_{0}\left(\ff(\xx_{0})-\ff(\xx^{\star})\right) + \omega_p(\xx^{\star},\zz_{0}) - \sum_{t \in [T]} A_t \lambda_t^{\frac{s}{s- p}}\lambda^{-\frac{p}{s-p}}
\,.
\end{equation}
Since $\ff(\xx_T) - \ff(\xx^{\star}) \geq 0$ and $\omega_p(\xx^{\star}, \zz_t) \geq 0$ for all $t$,
\[
\lambda^{-\frac{p}{s-p}}\sum_{t \in [T]} A_t \lambda_t^{\frac{s}{s- p}} \leq  \omega_p(\xx^{\star},\zz_{0})
= \Psi_0
\]
From \Cref{lem:ARelations}, $A_t^{1/p} \geq  \sum_{t \in [T]} \frac{1}{5 p \lambda_t^{1/p}}$. Applying Holder's inequality with $\alpha = \frac{1}{\nu}$ yields,
\begin{align*}
\sum_{t \in [T]} A_t^{\frac{1}{1 + \alpha}} 
	&= \sum_{t \in [T]} A_t^{\frac{1}{1 + \alpha}}  \lambda_t^{\frac{1}{p(1 + \alpha^{-1})}} \lambda_t^{-\frac{1}{p(1 + \alpha^{-1})}}
	\leq \left(
		\sum_{t \in [T]} A_t \lambda_t^{\frac{s}{s - p}}
	  \right)^{(1 + \alpha)^{-1}}
	   \left(
	  \sum_{t \in [T]} \lambda_t^{- \frac{1}{p}}
	  \right)^{(1 + \alpha^{-1})^{-1}} \\
	 &\leq
	\left( \lambda^{\frac{p}{s-p}}\Psi_0\right)^{\frac{1}{1 + \alpha}}
	 \left(5p A_T^{1/p}\right)^{\frac{\alpha}{1 + \alpha}}
\end{align*}
where we used that $\frac{\alpha}{1 + \alpha} = \frac{1}{1 + \alpha^{-1}}$ and $(1 + \alpha)^{-1} + (1 + \alpha^{-1})^{-1} = 1$.

We now use the above to give a lower bound on $A_T$. Using Lemma~\ref{lem:LBSeqA}, for $m = \alpha/p$, $B_j = A_t^{\frac{1}{1+\alpha}}$, $r_j = 1$, $\beta = \left((5p)^{\alpha}\lambda^{\frac{p}{s-p}}\Psi_0\right)^{-1/(1+\alpha)}$,
\[
A_T^{\frac{1}{1+\alpha}} \geq \left(\frac{\alpha-p}{\alpha}\left((5p)^{\alpha}\lambda^{\frac{p}{s-p}}\Psi_0\right)^{-1/(1+\alpha)}\cdot T \right)^{\frac{p}{\alpha-p}} = \left(\frac{p}{s} \left(\frac{1}{(5p)^{\alpha} \lambda^{\frac{p}{s-p}}\Psi_0}\right)^{\frac{1}{1+\alpha}} \cdot T\right)^{\frac{s-p}{p}}.
\]
Therefore 
\[
A_T \geq \left(\frac{1}{s}\right)^{\frac{s-p+sp}{p}} p^{\frac{s-p}{p}} \left(\frac{1}{\lambda^{\frac{p}{s-p}}\Psi_0}\right)^{\frac{s-p}{p}} \cdot T^{\frac{ps + s-p}{p}}.
\]
Further, using Eq.~\eqref{eq:PotHelp} and $\nu = \frac{1}{p} - \frac{1}{s}$,
\[
f(\xx_T) - \ff(\xx^{\star} ) \leq \frac{\Psi_0}{A_T} \leq \left(\frac{s^{\frac{s-p+sp}{p}}}{p^{\frac{s-p}{p}}}\right) \frac{\lambda \Psi_0^{s/p}}{T^{\frac{sp+s-p}{p}}} = \left(\frac{s^{s(1+\nu)}}{p^{s\nu}}\right) \frac{\lambda \Psi_0^{s/p}}{T^{s(1+\nu)}}.
\]

The result now follows from noting that $\Psi_0 = \|\xx^{\star}-\xx_0\|_p^p$

\end{proof}

\subsection{Proof of Theorem~\ref{thm:BallOracle}}

\begin{lemma}\label{lem:HelpBall}
Let $\ff : \R^d \rightarrow \R$ be continuously differentiable and strictly convex. For all $\cc \in \R^d$, define
\[
\xx = \argmin_{x \in \R^d | \|\xx-\cc\|_p\leq r}\ \ff(\xx).
\] 
Then $\xx$ is either the global minimizer of $\ff$, or $\|\xx-\cc\|_p = r$ and 

\[\nabla \ff(\xx) = -\frac{\|\nabla \ff(\xx)\|_{p/(p-1)}}{r^{p-1}} \diag{|\xx-\cc|}^{p-2}(\xx-\cc).
\]
\end{lemma}
\begin{proof}
The Lagrange dual gives,
\[
\min_{\yy}\max_{\lambda} \ff(\yy) + \frac{\lambda}{p}\left( r^p - \|\yy-\cc\|_p^p \right).
\]
There is some $\lambda\geq 0$ such that,
\[
\nabla\ff(\xx) = -\lambda \diag{|\xx-\cc|}^{p-2}(\xx-\cc).
\]
If $\lambda = 0$, then $\nabla\ff(\xx)=0$ and $\xx$ is the optimizer of $\ff$. If $\lambda>0$, then $\|\xx-\cc\|_p = r$, and as a result,
\[
\|\nabla\ff(\xx)\|_{p/(p-1)} = \lambda \|\xx-\cc\|_p^{p-1} = \lambda r^{p-1}.
\]
From the above we get, $\lambda = \frac{\|\nabla\ff(\xx)\|_{p/(p-1)} }{r^{p-1}}$, concluding the proof.
\end{proof}

We now state one final result from~\cite{carmon2022optimal} and then prove the main result of the section.
\begin{lemma}[Lemma 4,~\cite{carmon2022optimal}]\label{lem:LBSeqAInf}
Let $B_1, \cdots, B_k \in \mathbb{R}_{>0}$ be non-decreasing and, $r_1, \cdots, r_k\in \mathbb{R}_{>0}$ satisfy $B_i
 \geq \beta \sum_{j\in[i]} B_j r_j$  for some $\beta>0$, and all $i \in [k]$. Then $B_i \geq \exp\left(\beta \cdot \sum_{j\in [i]}r_j - 1\right)B_1$ for all $i \in [k]$.
\end{lemma}

\begin{proof}[Proof of Theorem~\ref{thm:BallOracle}]

 From Lemma~\ref{lem:HelpBall}, for $\lambda_t =\frac{\|\nabla\ff(\xx_{t+1})\|_{p/(p-1)} }{r^{p-1}}$, Lemma~\ref{lem:Potential} holds. Similar to the proof of Lemma~\ref{lem:Continuity} we can show that $\lambda_t$'s exist. Further, using Lemma~\ref{lem:Potential} and the fact that $\|\xx_{t+1}-\yy_t\|_p = r$ (Lemma~\ref{lem:HelpBall}), implies that for all $T \geq 1$
\begin{equation}\label{eq:PotHelpInf}
A_{T}\left(\ff(\xx_{T})-\ff(\xx^{\star})\right) + \omega_p(\xx^{\star},\zz_{T}) \leq A_{0}\left(\ff(\xx_{0})-\ff(\xx^{\star})\right) + \omega_p(\xx^{\star},\zz_{0}) - r^p\sum_{t \in [T]} A_t \lambda_t.
\end{equation}
Since $\ff(\xx_T) - \ff(\xx^{\star}) \geq 0$ and $\omega_p(\xx^{\star}, \zz_t) \geq 0$ for all $t$. This implies that 
\[
r^p\sum_{t \in [T]} A_t \lambda_t \leq  \omega_p(\xx^{\star},\zz_{0})
= \Psi_0
\]
Since, $A_t^{1/p} \geq  \sum_{t \in [T]} \frac{1}{5 p \lambda_t^{1/p}}$  by Lemma~\ref{lem:ARelations}, applying Holder's inequality for some $\alpha $ yields that
\begin{align*}
\sum_{t \in [T]} A_t^{\frac{1}{1 + \alpha}} 
	&= \sum_{t \in [T]} A_t^{\frac{1}{1 + \alpha}}  \lambda_t^{\frac{1}{p(1 + \alpha^{-1})}} \lambda_t^{-\frac{1}{p(1 + \alpha^{-1})}}
	\leq \left(
		\sum_{t \in [T]} A_t \lambda_t^{\alpha/p}
	  \right)^{(1 + \alpha)^{-1}}
	   \left(
	  \sum_{t \in [T]} \lambda_t^{- \frac{1}{p}}
	  \right)^{(1 + \alpha^{-1})^{-1}} \\
	 &\leq
	\left( r^{-p}\Psi_0\right)^{\frac{1}{1 + \alpha}}
	 \left(5p A_T^{1/p}\right)^{\frac{\alpha}{1 + \alpha}}
\end{align*}
where we used that $\frac{\alpha}{1 + \alpha} = \frac{1}{1 + \alpha^{-1}}$ and $(1 + \alpha)^{-1} + (1 + \alpha^{-1})^{-1} = 1$.

We now use the above to give a lower bound on $A_T$. Using Lemma~\ref{lem:LBSeqAInf}, for $\alpha = p$,  $B_j = A_t^{\frac{1}{1+\alpha}}$, $r_j = 1$, $\beta = \left((5p)^{\alpha}r^{-p}\Psi_0\right)^{-1/(1+\alpha)}$,
\[
A_T^{\frac{1}{1+\alpha}} \geq \exp\left(\left((5p)^{\alpha}r^{-p}\Psi_0\right)^{-1/(1+\alpha)}\cdot T -1 \right).
\]
Therefore, when $T = O\left(p^{p/(p+1)} \left(\frac{\Psi_0^{1/p}}{r}\right)^{\frac{p}{p+1}}\log \frac{\Psi_0}{\epsilon}\right)$ 
\[
A_T \geq \frac{\Psi_0}{\epsilon}.
\]
Further, using Eq.~\eqref{eq:PotHelpInf}
\[
f(\xx_T) - \ff(\xx^{\star} ) \leq \frac{\Psi_0}{A_T} \leq  \epsilon.
\]
Additionally noting that $\Psi_0 = \|\xx^{\star}-\xx_0\|_p^p$ we get that in $T = O\left(p^{p/(p+1)} \left(\frac{\|\xx_0-\xx^{\star}\|_p}{r}\right)^{\frac{p}{p+1}}\log \frac{\Psi_0}{\epsilon}\right)$ iterations, we obtain an $\epsilon$-approximate solution.

\end{proof}

\section{Proofs for Line-Search Free Method}\label{app:LSF}

\subsection{Reductions from the \texorpdfstring{$\sigmaGMS$}{TEXT} Oracle}

We first show a reduction to a general $\sigmaMS$ oracle. This is the same as the MS oracles from~\cite{monteiro2013accelerated,carmon2022optimal} when $p=2$. 

\begin{definition}
    An oracle $\Oracle : \mathbb{R}^d \rightarrow \R^d \times \R_+$ is a $\sigmaMS$ oracle for a function $f : \R^d \rightarrow \R$ if, for every $\cc \in \R^d$, $(\xx, \gamma) = \Oracle(\cc)$
    satisfies
    \begin{equation*}
        \norm{\gamma\diag{|\xx-\cc|^{p-2}}(\xx - \cc) + \nabla f(\xx))}_{p^*} \leq \sigma \gamma\norm{\xx - \cc}^{p-1}_p.
    \end{equation*}
    Additionally, we say $\Oracle$ satisfies an $(s,\mu)$ movement bound if 
    \begin{equation*}
        \norm{\xx -  \cc}_p  \geq \begin{cases}
            (\gamma/\mu^s)^{1/(s-p)} & \text{ if $s<\infty$}\\
            1/\mu & \text{ if $s = \infty$}.
        \end{cases}
    \end{equation*}
\end{definition}

We now show how the $\sigmaMS$ oracle is related to the $\sigmaGMS$ oracle.
\begin{lemma}\label{lem:MS-oracle}
    If $(\xx, \gamma) = \sigmaMS(y)$, then 
    \begin{align}
        \langle\nabla \ff(\xx),\xx -\yy  \rangle
        \leq -\gamma(1-\sigma)\norm{\xx -\yy}_p^p, \text{ and } \norm{\nabla \ff(\xx)}_{p^*} \leq (1+\sigma) \gamma \norm{\xx-\yy}_p^{p-1}. 
        \end{align}
\end{lemma}
\begin{proof}
    Observe that
    \begin{align*}
        \langle&\nabla \ff(\xx),\xx -\yy  \rangle\\
        &= \langle\nabla \ff(\xx) + \gamma\diag{|\xx-\cc|^{p-2}}(\xx -\yy) - \gamma\diag{|\xx-\cc|^{p-2}}(\xx -\yy),\xx -\yy \rangle\\
        & \leq \langle\nabla \ff(\xx) + \gamma\diag{|\xx-\cc|^{p-2}}(\xx -\yy), \xx -\yy \rangle - \gamma\|\xx-\yy\|_p^p\\
        &\leq \norm{\nabla \ff(\xx) + \gamma\diag{|\xx-\cc|^{p-2}}(\xx -\yy)}_{p^*}\norm{\xx -\yy}_p - \gamma\|\xx-\yy\|_p^p\\
        &\leq \sigma\gamma\norm{\xx -\yy}_p^p -  \gamma\|\xx-\yy\|_p^p,
        \end{align*}
        as required. The last inequality above follows from the definition of a $\sigma$-$\textrm{MS}_p$ oracle. We now show the bound on $\|\nabla\ff(\xx)\|_{p^*}$.
        \begin{align*}
            \|\nabla\ff(\xx)\|_{p^*}& = \|\nabla\ff(\xx) + \gamma\diag{|\xx-\cc|^{p-2}}(\xx - \cc) - \gamma\diag{|\xx-\cc|^{p-2}}(\xx - \cc) +\|_{p^*}\\
            & \leq \|\nabla\ff(\xx) + \gamma\diag{|\xx-\cc|^{p-2}}(\xx - \cc)\|_{p^*} + \gamma\|\diag{|\xx-\cc|^{p-2}}(\xx - \cc) +\|_{p^*}\\
            & \leq \sigma \gamma \|\xx-\cc\|_p^{p-1} + \gamma \|\xx-\cc\|_p^{p-1},
        \end{align*}
        where the last step follows from the definition of $p^*$ and the $\sigmaMS$ oracle.
\end{proof}

We now prove that the output of an $\ell_p^s$ oracle is a $0$-$\textrm{GMS}_p$ oracle.
\begin{lemma}
    An oracle that returns $(\xx, \gamma)$ with input $\cc\in \rea^d, \ff:\rea^d\rightarrow \rea$, and 
    \[
    \xx = \text{ the output of a $\ell_p^s(\lambda)$-proximal oracle for function $f$ queried at $\cc$ and,}
    \gamma = \begin{cases}
    \lambda s\|\xx-\cc\|_p^{s-p} & \text{ if $s<\infty$}\\
    \frac{\|\nabla\ff(\xx)\|_{p^*}}{r^{p-1}} & \text{ if $s = \infty$},
    \end{cases}
    \]
    is a $0$-$\textrm{GMS}_p$ oracle that satisfies a $(s,\min\{1/r,(\lambda s)^{1/s}\})$ movement bound.
\end{lemma}
\begin{proof}
Let $\xx$ be an output of the oracle with input $\cc,\ff$. From the optimality conditions for $s<\infty$,
\[
\nabla\ff(\xx) = -\lambda s \|\xx-\cc\|_p^{s-p} \diag{|\xx-\cc|^{p-2}}(\xx-\cc),
\]
and for $s = \infty$ (from Lemma~\ref{lem:HelpBall}),
\[
\nabla \ff(\xx) = -\frac{\|\nabla\ff(\xx)\|_{p^*}}{r^{p-1}}\diag{|\xx-\cc|^{p-2}}(\xx-\cc).
\]
Therefore, in both cases, $\langle \nabla \ff(\xx),\xx-\cc\rangle = -\gamma \|\xx-\cc\|_p^p$. Furthermore, in both cases by a simple computation we see that $\|\nabla\ff(\xx)\|_{p*} = \gamma \|\xx-\cc\|_p^{p-1},$ as required. We next show the movement bounds. For $s=\infty$, from Lemma~\ref{lem:HelpBall} $\|\xx-\cc\|_p = r$, and for $s<\infty$,
\[
\gamma^{\frac{1}{s-p}}\mu^{-\frac{s}{s-p}} = (\lambda s)^{\frac{1}{s-p}}\mu^{-\frac{s}{s-p}} \|\xx-\cc\|_p.
\]
The bound now follows from requiring $\mu \geq 1/r$ and $\mu\geq (\lambda s)^{1/s}$.
\end{proof}

With the above result on $\ell_p^s(\lambda)$-oracles in hand, Theorems~\ref{thm:ConceptualProx-main} and~\ref{thm:Ball-main} follow as corollaries of Theorems~\ref{lem:finite-s-lemma} and~\ref{lem:Inf-s-lemma} respectively.

\subsection{Proof of Lemma~\ref{lem:Potential2}}
\begin{proof}
Observe,
\begin{align*}
   & A_{t+1}' \left(\ff(\xx_{t+1}')-\ff(\xx^{\star})\right) - A_{t}\left(\ff(\xx_{t})-\ff(\xx^{\star})\right) \\
   &= a_{t+1}\left(\ff(\xx_{t+1}')-\ff(\xx^{\star})\right) + A_t \left(\ff(\xx_{t+1}')-\ff(\xx_t)\right)\\
   &\leq  a_{t+1}\langle \nabla \ff(\xx_{t+1}'),\xx_{t+1}'-\xx^{\star}\rangle + A_t\langle \nabla \ff(\xx_{t+1}'),\xx_{t+1}'-\xx_t\rangle &&\text{(convexity of $\ff$)}\\
   &= \langle \nabla \ff(\xx_{t+1}'),A_{t+1}'\xx_{t+1}' - A_t\xx_t - a_{t+1}\xx^{\star}\rangle\\
   &= \langle \nabla \ff(\xx_{t+1}'),A_{t+1}'\xx_{t+1}' +a_{t+1}\zz_t-A_{t+1}'\yy_t  - a_{t+1}\xx^{\star}\rangle &&\text{(Since $A_{t+1} y_t = A_t x_t + a_{t+1} z_t$)}\\
   &= A_{t+1}'{\langle\nabla \ff(\xx_{t+1}'),\xx_{t+1}' -\yy_t  \rangle} + a_{t+1} {\langle \nabla \ff(\xx_{t+1}'),\zz_{t}  - \xx^{\star}\rangle} 
\end{align*}
Using the KKT conditions for $z_t$, i.e., differentiating the equation defining $z_t$ in the algorithm w.r.t $z$, we have
\[
\sum_{i=1}^t a_i \beta_i \nabla \ff(\xx_i') = - \nabla \omega_p(\zz_t,\yy_0), \quad \forall t.
\]
Therefore, 
\[
a_{t+1} \beta_{t+1} \nabla\ff(\xx_{t+1}') = \nabla \omega_p(\zz_t,\yy_0) - \nabla \omega_p(\zz_{t+1},\yy_0)
\]
This implies
\begin{align*}
a_{t+1} \beta_{t+1} \langle \nabla \ff(\xx_{t+1}'),\zz_{t+1}  - \xx^{\star}\rangle &= \langle \nabla \omega_p (\zz_t, \yy_0) - \nabla \omega_p(\zz_{t+1}, \yy_0), \zz_{t+1} - \xx^\star \rangle \\
&= \omega_p(\xx^{\star},\zz_{t}) - \omega_p(\xx^{\star},\zz_{t+1}) - \omega_p(\zz_{t},\zz_{t+1}) \\
&\leq \omega_p(\xx^{\star},\zz_{t}) - \omega_p(\xx^{\star},\zz_{t+1}) - \frac{1}{2^{p+1}} \norm{\zz_{t+1} - \zz_t}_p^p
\end{align*}
where the second equality is the Bregman three-point property (Lemma~\ref{lem:3Point}) and the inequality uses Lemma \ref{lem:PrecAKPS}. We additionally bound

\begin{align*}
a_{t+1} \langle \nabla \ff(\xx_{t+1}'),\zz_{t+1}  - \zz_t \rangle &\leq a_{t+1} \norm{\nabla \ff(\xx_{t+1}')}_{p^*} \norm{\zz_{t+1} - \zz_t}_p
\end{align*}
by H{\"o}lder's inequality. Combining these inequalities yields

\begin{align*}
a_{t+1} \langle \nabla \ff(\xx_{t+1}'),\zz_{t}  - \xx^\star \rangle &\leq \beta_{t+1}^{-1}(\omega_p(\xx^{\star},\zz_{t}) - \omega_p(\xx^{\star},\zz_{t+1}) - \frac{1}{2^{p+1}} \norm{\zz_{t+1} - \zz_t}_p^p)\\
&+  a_{t+1} \norm{\nabla \ff(\xx_{t+1}')}_{p^*} \norm{\zz_{t+1} - \zz_t}_p \\
&\leq  \beta_{t+1}^{-1} (\omega_p(\xx^{\star},\zz_{t}) - \omega_p(\xx^{\star},\zz_{t+1})) + \frac{p-1}{p} (2a_{t+1} \beta_{t+1}^{1/p})^{p^*} \norm{\nabla \ff(\xx_{t+1}')}_{p^*}^{p^*}\\
&+ \frac{1}{\beta_{t+1} p 2^p} \norm{\zz_{t+1} - \zz_t}_p^p - \frac{1}{\beta_{t+1} 2^{p+1}} \norm{\zz_{t+1} - \zz_t}_p^p \\
&\leq  \beta_{t+1}^{-1} (\omega_p(\xx^{\star},\zz_{t}) - \omega_p(\xx^{\star},\zz_{t+1}) ) + (2a_{t+1} \beta_{t+1}^{1/p})^{p^*} \norm{\nabla \ff(\xx_{t+1}')}_{p^*}^{p^*}
\end{align*}
where the second inequality uses Young's inequality. 
Finally, since $\xx'_{t+1}$ is returned by $\sigmaGMS(\yy_t)$, we have
\begin{equation}\label{eq:HelpNLS}
\langle\nabla \ff(\xx_{t+1}'),\xx_{t+1}' -\yy_t  \rangle \leq  - (1-\sigma)\lambda_{t+1}\norm{\xx_{t+1}' - \yy_t}_p^p
\end{equation}
and 
    \begin{equation}
        \norm{\nabla \ff(\xx_{t+1}')}_{p^*} \leq (1+\sigma) \lambda_{t+1} \norm{\xx_{t+1}'-\yy_t}_p^{p-1}.
    \end{equation}
Putting everything together, we obtain
\begin{align*}
&A_{t+1}' \left(\ff(\xx_{t+1}')-\ff(\xx^{\star})\right) - A_{t}\left(\ff(\xx_{t})-\ff(\xx^{\star})\right) \\
&\leq A_{t+1}'{\langle\nabla \ff(\xx_{t+1}'),\xx_{t+1}' -\yy_t  \rangle} + a_{t+1} {\langle \nabla \ff(\xx_{t+1}'),\zz_{t}  - \xx^{\star}\rangle} \\
&\leq - A_{t+1}'(1-\sigma) \lambda_{t+1} \norm{\xx_{t+1}' - \yy_t}_p^p + (2a_{t+1} \beta_{t+1}^{1/p})^{p^*} \norm{\nabla \ff(\xx_{t+1}')}_{p^*}^{p^*} + \beta_{t+1}^{-1} \left( \omega_p(\xx^{\star},\zz_{t}) - \omega_p(\xx^{\star},\zz_{t+1}) \right) \\
&\leq \left(((2+2\sigma) a_{t+1} \lambda_{t+1} \beta_{t+1}^{1/p})^{p^*} - A_{t+1}' (1-\sigma) \lambda_{t+1} \right) \norm{\xx_{t+1}' - \yy_t}_p^p + \beta_{t+1}^{-1} \left( \omega_p(\xx^{\star},\zz_{t}) - \omega_p(\xx^{\star},\zz_{t+1}) \right).
\end{align*}

We now consider the cases where $\beta_{t+1} = 1$ and $\beta_{t+1} < 1$ separately. If $\beta_{t+1} = 1$ we must have $\lambda_{t+1} \leq \lbar$, $A_{t+1} = A_{t+1}'$, and $\xx_{t+1} = \xx_{t+1}'$. The above therefore yields 

\begin{align*}
&A_{t+1} \left(\ff(\xx_{t+1})-\ff(\xx^{\star})\right) - A_{t}\left(\ff(\xx_{t})-\ff(\xx^{\star})\right) \\
&\leq \lambda_{t+1} \left(((2 + 2\sigma) a_{t+1})^{p^*} \lambda_{t+1}^{\frac{1}{p-1}} - A_{t+1} (1-\sigma) \right) \norm{\xx_{t+1} - \yy_t}_p^p + \omega_p(\xx^{\star},\zz_{t}) - \omega_p(\xx^{\star},\zz_{t+1}).
\end{align*}
Now by the definition of $A_{t+1}$, we have
\[
A_{t+1} = \left(\frac{3 (1+\sigma)}{1-\sigma} a_{t+1}\right)^{p^*} \lbar^{\frac{1}{p-1}} \ge \frac{1}{1-\sigma} ((2+2\sigma) a_{t+1})^{p^*} \lambda_{t+1}^{\frac{1}{p-1}} 
\]
and so we have 
\[
A_{t+1} \left(\ff(\xx_{t+1})-\ff(\xx^{\star})\right) + \omega_p(\xx^{\star},\zz_{t+1}) \leq A_{t}\left(\ff(\xx_{t})-\ff(\xx^{\star})\right) + \omega_p(\xx^{\star},\zz_{t}).
\]

By an equivalent argument to \Cref{lem:ARelations}, we observe $A_{t+1}^{1/p} \geq A_t^{1/p} + \frac{1}{\frac{3+3\sigma}{1-\sigma} p^{1/p} \lbar^{1/p}} \geq A_t^{1/p} + \frac{1}{\frac{6+6\sigma}{1-\sigma}
\lbar^{1/p}}$ (as $p^{1/p} \leq 2$ for all $p$), which combined with the above potential bound yields one case of the lemma.

If $\beta_{t+1} < 1$, we have $\lambda_{t+1} \geq \lbar$ and $\beta_{t+1} \lambda_{t+1} = \lbar$. We observe, again using Eq.~\eqref{eq:HelpNLS},

\begin{align*}
&A_{t+1}  \left(\ff(\xx_{t+1})-\ff(\xx^{\star})\right) -  A_t \left(\ff(\xx_{t})-\ff(\xx^{\star})\right) \\
&\leq \beta_{t+1} A'_{t+1}  \left(\ff(\xx_{t+1}')-\ff(\xx^{\star})\right) + (1-\beta_{t+1}) A_t \left(\ff(\xx_{t})-\ff(\xx^{\star})\right) -  A_t \left(\ff(\xx_{t})-\ff(\xx^{\star})\right) \\
&= \beta_{t+1} A_{t+1}' \left(\ff(\xx_{t+1}')-\ff(\xx^{\star})\right) - \beta_{t+1} A_{t}\left(\ff(\xx_{t})-\ff(\xx^{\star})\right) \\
&\leq \beta_{t+1} \left(( (2+2\sigma) a_{t+1} \lambda_{t+1} \beta_{t+1}^{1/p})^{p^*} - A_{t+1}' (1-\sigma) \lambda_{t+1} \right) \norm{\xx_{t+1}' - \yy_t}_p^p +  \omega_p(\xx^{\star},\zz_{t}) - \omega_p(\xx^{\star},\zz_{t+1}) \\
&=  \left(((2+2\sigma)a_{t+1})^{p^*} \lambda_{t+1}^{p^*} \beta_{t+1}^{p^*} - A_{t+1}' (1-\sigma) \beta_{t+1} \lambda_{t+1} \right) \norm{\xx_{t+1}' - \yy_t}_p^p +  \omega_p(\xx^{\star},\zz_{t}) - \omega_p(\xx^{\star},\zz_{t+1}) \\
&= \left(((2+2\sigma) a_{t+1})^{p^*} \lbar^{p^*} - A_{t+1}' (1-\sigma) \lbar \right) \norm{\xx_{t+1}' - \yy_t}_p^p +  \omega_p(\xx^{\star},\zz_{t}) - \omega_p(\xx^{\star},\zz_{t+1})
\end{align*}
where the last equality uses the definition of $\beta_{t+1}$ and the first inequality uses convexity of $f$, the definition of $x_{t+1}$, and that $A_{t+1} = \beta_{t+1} A_{t+1}' + (1-\beta_{t+1}) A_t = A_t + \beta_{t+1} a_{t+1}$. The result now follows from $A_{t+1}' = (\frac{3+3\sigma}{1-\sigma} a_{t+1})^{p^*} \lbar^{\frac{1}{p-1}}$.
\end{proof}
\subsection{Proof of Lemma~\ref{lem:finite-s-lemma}}

\begin{proof}
We first observe that 
\[
\ff(\xx_{T+1}) - \ff(\xx^\star)  \leq \frac{\Psi_0}{A_{T+1} }.
\]
We will thus bound the right-hand side expression. We bound the number of iterations required for $A_t$ to increase by a factor of $2$. Let $T_1 \geq 1$ be given and let $T_2$ be the smallest value where $A_{T_2} \geq 2 A_{T_1}$. We will categorize these iterations by whether $\lbar\geq\lambda_{t+1}$ or $\lbar< \lambda_{t+1}$.  Let $S$, $L$ be the sets of iterations where $\lbar\geq\lambda_{t+1}$ and  $\lbar< \lambda_{t+1}$ respectively and note that $S \cup L = [T_1, T_2]$.

    Summing the bound of Lemma~\ref{lem:Potential2} over all iterations and recalling $A_0 = 0$ gives
\[
A_{T+1} (\ff(\xx_{T+1}) - \ff(\xx^\star)) + \omega_p(\xx^\star, \zz_{t+1}) \leq  \omega_p(\xx^\star, \xx_0) - \sum_{t \in L} \frac{1-\sigma}{3} A_{t+1}' \lbar \|\xx'_{t+1}-\yy_t\|_p^p.
\]
Since the oracle satisfies the $(s,\mu)$ movement bounds for all $t$, $\|\xx'_{t+1}-\yy_t\|_p^p \geq \lambda_{t+1}^{\frac{p}{s-p}}\mu^{-\frac{sp}{s-p}}$. Further since for all $t\in L$, $\lambda_{t+1}>\lbar$, the above becomes, 
\[
\|\xx'_{t+1}-\yy_t\|_p^p \geq \lbar^{\frac{p}{s-p}}\mu^{-\frac{sp}{s-p}}.
\]
Since $\ff(\xx_T) - \ff(\xx^{\star}) \geq 0$ and $\omega_p(\xx^{\star}, \zz_t) \geq 0$ for all $t$. This implies that 
\[
\frac{1-\sigma}{3} \mu^{-\frac{sp}{s-p}}\sum_{t \in L} A_{t+1}' \lbar^{\frac{s}{s- p}} \leq  \omega_p(\xx^{\star},\zz_{0})
= \Psi_0
\]
We now observe, by our choice in Algorithm~\ref{alg:LSFProx}, that $\bar{\lambda}_{t+1} = \bar{\lambda} = A_{T_1}^{-\frac{(s-p)(p+1)}{ps-p+s}} \Psi_0^{\frac{p(s-p)}{ps-p+s}} \mu^{\frac{sp^2}{s-p+ps}}$ for all $t \in [T_1,T_2]$. Since $A_{t+1}' \geq A_{T_1}$, this implies

\[
\frac{1-\sigma}{3} \mu^{-\frac{sp}{s-p}} A_{T_1} \bar{\lambda}^{\frac{s}{s- p}} \cdot  |L| \leq \Psi_0,
\]
or,
\[
|L| \leq \frac{3}{1-\sigma} \mu^{\frac{sp}{s-p}}\bar{\lambda}^{-\frac{s}{s- p}}\frac{\Psi_0}{A_{T_1}} .
\]
We will now bound the size of set $S$. Since for every $t \in S$ we have $A_{t+1}^{1/p} \geq A_t^{1/p} + \frac{1}{\frac{6+6\sigma}{1-\sigma} \lbar^{1/p}} = A_t^{1/p} + \frac{1}{\frac{6+6\sigma}{1-\sigma} \bar{\lambda}^{1/p}}$, we conclude
\[
2^{1/p}A_{T_1}^{1/p} \geq A_{T_2-1}^{1/p} \geq A_{T_1}^{1/p} + \frac{|S|}{\frac{6+6\sigma}{1-\sigma}  \bar{\lambda}^{1/p}} \Rightarrow |S| < \frac{6+6\sigma}{1-\sigma}  (A_{T_1} \bar{\lambda})^{1/p}.
\]
Summing these gives 
\[
T_2 - T_1 = |S| +|L| \leq  \frac{6+6\sigma}{1-\sigma}  (A_{T_1} \bar{\lambda})^{1/p} + \frac{3}{1-\sigma} \mu^{\frac{sp}{s-p}}\frac{\Psi_0 }{A_{T_1} \bar{\lambda}^{\frac{s}{s-p}}} .
\]
By our choice of $\bar{\lambda}$, we have
\[
T_2 - T_1 \leq \frac{24}{1-\sigma} A_{T_1}^\frac{p}{ps-p+s} \Psi_0^{\frac{s-p}{ps-p+s}} \mu^{\frac{sp}{s-p+ps}}.
\]
We now sum this bound over all phases encountered in the $T$ iterations. The sum over phases is the sum of a geometric series with initial term $ O\left( A_{1}^\frac{p}{ps-p+s} \Psi_0^{\frac{s-p}{ps-p+s}} \mu^{\frac{sp}{s-p+ps}}\right)$ and common ratio $2^\frac{p}{ps-p+s}$. We now use the following property of a geometric series with $n$ terms, starting term $a$ and common ratio $r$:
\[
\sum_{i=0}^{n-1} ar^i = a\left(\frac{r^{n}-1}{r-1}\right) \leq ar^{n-1} \left(\frac{r}{r-1}\right).
\]
Note that $ar^{n-1}$ is the last term of the series. Applying this to our series, since $\frac{r}{r-1} = O(1)$, we get, 
\[
T \leq O\left( A_{T+1}^\frac{p}{ps-p+s} \Psi_0^{\frac{s-p}{ps-p+s}} \mu^{\frac{sp}{s-p+ps}}\right). 
\]
which implies
\[
\Psi_0 A_{T+1}^{-1} \leq O \left( \frac{\mu^{s} \Psi_0^{s/p}}{T^{(ps-p+s)/p}} \right).
\]
The result follows with the observation that $\frac{ps-p+s}{p} = s (1 + \nu)$ for $\nu = \frac{1}{p} - \frac{1}{s}$.

In addition to the cost of querying the $\sigmaGMS$ oracle, the remaining computational costs are bounded by the costs to update $y_t$, $x_{t+1}$, and $z_{t+1}$ (which can be expressed in closed-form as a per-coordinate scaling of the (weighted) accumulation of previous gradients), which are $O(d)$. In addition, $\delta$-approximately finding a positive root of the polynomial $m(a)=\lbar^{1/p-1} (\frac{3+3\sigma}{1-\sigma} a)^{p/(p-1)} - a - A_t = 0$ (i.e., $\tilde{a}$ such that $|\tilde{a} - a_{\mathrm{root}}| \leq \delta$), is bounded, for $u = (\frac{1-\sigma}{3+3\sigma})\frac{A_t^{(p-1)/p}}{\lbar^{1/p}} + \lbar^{-(p-1)}(\frac{3+3\sigma}{1-\sigma})^{-p}$, by $O(\log(u/\delta))$, since $m(0) = -A_t \leq 0$ and $m(u) \geq 0$, though we note that this cost is dominated by $O(d)$.
\end{proof}

\subsection{Proof of Lemma~\ref{lem:Inf-s-lemma}}

\begin{proof}
  We follow the proof of the previous lemma. Define $S$ and $L$ as in the proof of Lemma~\ref{lem:finite-s-lemma}. Again, summing the bound of \Cref{lem:Potential2} over all iterations and recalling $A_0 = 0$ gives
\[
A_{T+1} (\ff(\xx_{T+1}) - \ff(\xx^\star)) + \omega_p(\xx^\star, \zz_{t+1}) \leq  \omega_p(\xx^\star, \xx_0) - \frac{1-\sigma}{3}  \sum_{t \in L}  A_{t+1}' \lbar\|\xx'_{t+1}-\yy_t\|_p^p.
\]
Using the fact that the oracle satisfies the $(\infty,\mu)$ movement bounds, we know that $\|\xx'_{t+1}-\yy_t\|_p^p \geq 1/\mu^p$ for all t.

Since $\ff(\xx_T) - \ff(\xx^{\star}) \geq 0$ and $\omega_p(\xx^{\star}, \zz_t) \geq 0$ for all $t$. This implies that 
\[
\frac{1-\sigma}{3} \mu^{-p} \sum_{t \in L} A_{t+1}' \lbar \leq  \omega_p(\xx^{\star},\zz_{0})
= \Psi_0
\]
We now choose $\bar{\lambda}_t = \bar{\lambda}$ for all $t \in [T_1,T_2]$. We will choose this value $\bar{\lambda}$ in the end. Since $A_{t+1}' \geq A_{T_1}$, this implies

\[
|L| \leq \frac{3}{1-\sigma} \frac{\mu^p\Psi_0}{\bar{\lambda} A_{T_1}}.
\]
The bound on $|S|$ is identical to Lemma~\ref{lem:finite-s-lemma}: 
\[
 |S| < \frac{6+6\sigma}{1-\sigma} (A_{T_1} \bar{\lambda})^{1/p}.
\]
Summing these gives 
\[
T_2 - T_1 = |S| +|L| \leq  \frac{6+6\sigma}{1-\sigma} (A_{T_1} \bar{\lambda})^{1/p} + \frac{3}{1-\sigma} \frac{\mu^p \Psi_0}{\bar{\lambda}A_{T_1}} .
\]
Setting 
\[
\bar{\lambda} = \frac{\mu^{\frac{p^2}{p+1}}\Psi_0^{\frac{p}{p+1}}}{A_{T_1} },
\]
for these iterations yields 
\[
T = T_2 - T_1 \leq \frac{24}{1-\sigma} \mu^{\frac{p}{p+1}} \Psi_0^\frac{1}{p+1}.
\]
Now, in every $T$ iterations, the value of $A_{T}$ doubles. Therefore, we need a total of 
\[
T = O\left(  \mu^{\frac{p}{p+1}} \Psi_0^\frac{1}{p+1} \log\frac{\Psi_0}{\epsilon A_0}\right)
\]
iterations to get $f(\xx_T) - \ff(\xx^{\star})\leq \frac{\Psi_0}{A_T} \leq \epsilon$.
\end{proof}

\section{Proofs for Applications}\label{app:app}

\subsection{Proof of Theorem~\ref{thm:regression-main}}

Leveraging our framework from Section~\ref{sec:NoLS} for the $\ell_s$-regression problem, we get as an immediate corollary of Theorem~\ref{thm:ConceptualProx-main},
\begin{corollary}\label{lem:ProxSNorm}
   Algorithm~\ref{alg:LSFProx} applied to 
   $\ff(\xx) =  \|\AA\xx-\bb\|_s^s$ finds $\xxtil$ such that for all $k$, and $\xx^{\star}:=\arg\min \ff(\xx)$
   \[
   \ff(\xxtil)-\ff(\xx^{\star})\leq \frac{s\|\AA\xx^{\star}-\AA\xx_0\|_p^s}{k^{\frac{s(p+1)-p}{p}}}
   \]
   Each iteration $k$ involves solving a proximal subproblem of the form,
   \[
   \min_{x} f(x) + \lambda_k \|A(x-c_k)\|_p^s,
    \] for given constant $\lambda_k$ and vector $c_k$. 
\end{corollary}

We would use this corollary as the basis to prove our result. We will prove a bound on $\|\AA\xx^{\star}-\AA\xx_0\|_p^s$ and show how to solve every proximal subproblem efficiently using $\tilde{O}(1)$ smoothed $\ell_p$-regression problems. We begin by proving a bound on $\|\AA\xx^{\star}-\AA\xx_0\|_p^s$.

\begin{lemma}\label{lem:RelateDivergenceFunc}
For any $\AA\in \mathbb{R}^{n\times d}$, $2\leq p\leq s$, let $\ff(\xx) = \|\AA\xx-\bb\|_s^s$ and $\xx^{\star}= \arg\min_{\xx} \ff(\xx)$. Then,
\[
\norm{A (x - x^*)}_p \leq 2^{1+1/s} n^{\nu}\cdot \left(f(x) - f(x^*) \right)^{1/s}.
\]
\end{lemma}
\begin{proof}
    Let $x\in \mathbb{R}^d$ be fixed and let $\calE :=\ff(\xx) - \ff(\xx^{\star})$. We first note from Lemma~\ref{lem:PrecAKPS},
    \[
    \ff(\xx) - \ff(\xx^{\star}) - \langle\nabla \ff(\xx^{\star}),\xx-\xx^{\star}\rangle \geq \frac{1}{2^{s+1}}\|\AA(\xx-\xx^{\star})\|_s^s,
    \]
    and since $\nabla \ff(\xx^{\star}) = 0$,
    \[
    \|\AA(\xx-\xx^{\star})\|_s\leq 2^{1+1/s}\left( \ff(\xx) - \ff(\xx^{\star})\right)^{1/s}= 2^{1+1/s}\calE^{1/s}.
    \]
   Using the relation between norms, i.e., for $s\geq p,$ and any $z\in \mathbb{R}^n$, $\|z\|_s\leq \|z\|_p\leq n^{\frac{1}{p}-\frac{1}{s}} \|z\|_s$ we can further bound,
    \begin{equation*}\label{eq:boundpNorm}
    \|\AA(\xx-\xx^{\star})\|_p \leq n^{\frac{1}{p}-\frac{1}{s}}\cdot 2^{1+1/s}\calE^{1/s},
    \end{equation*}
    as required.
\end{proof}

\paragraph{Solving proximal problems efficiently.}

We require solving proximal problems of the form,
\[
\min_{\xx} \|\AA\xx-\bb\|_s^s+ \lambda\|\AA(\xx-\xx_t)\|_{p}^s
\]
We will prove the following.
\begin{lemma}\label{lem:proxSubProb} There is an algorithm that can solve the $\ell_p^s(\lambda)$-proximal point problem using $\Otil(1)$ problems of the form 
\[\min_{\xx} \dd^{\top}\xx + \|\xx-\xx_t\|_{\nabla^2\ff(\xx_t)}^2 + O(\lambda_t)\|\AA(\xx-\xx_t)\|_p^p
\]
\end{lemma}
We first bound the Hessian of the prox problem.
\begin{lemma}\label{lem:prec}
    Let $s,p\geq 2$ and define $\ff(\xx) =  \|\AA\xx-\bb\|_s^s$. For any $\yy$, define $\ff_{\yy}(\xx) = \ff(\xx) + C_s\|\AA(\yy-\xx)\|_{p}^s$ and $\hh_{\yy}(\xx) = \|\xx-\yy\|^2_{\nabla^2 \ff(\yy)}  + C_s\|\AA(\yy-\xx)\|_{p}^s$. Then for $C_s = e\cdot s^s$, for any $\xx$,
    \[
    \frac{1}{e}\nabla^2\hh_{\yy}(\xx) \preceq \nabla^2\ff_{\yy}(\xx) \preceq e\cdot \nabla^2 \hh_{\yy}(\xx).
    \]
\end{lemma}
\begin{proof}
    This result follows from a small tweak to the proof of Lemma 4.3 of~\cite{jambulapati2022improved}. We include the proof here for completeness. Observe that,
    \[
    \nabla^2\ff(\xx ) = s(s-1)\AA^{\top}\diag{|\AA\xx-\bb|}^{s-2}\AA.
    \]
    For any vector $\zz$,
    \begin{align*}
        \zz^{\top}\nabla^2\ff(\xx)\zz & =  s(s-1)\sum_{i\in [n]} |\AA\xx-\bb|_i^{s-2}(\AA\zz)_i^2\\
        & = s(s-1)\sum_{i\in [n]} |\AA\yy-\bb + \AA(\xx-\yy)|_i^{s-2}(\AA\zz)_i^2\\
        & \leq  \sum_i \left(es(s-1) |\AA\yy-\bb|_i^{s-2} + s^s|\AA(\xx-\yy)|_i^{s-2}\right)(\AA\zz)_i^2,
    \end{align*}
    where the last line follows from Lemma 4.1 of~\cite{jambulapati2022improved}.
    We will now use Hölder's inequality and the fact that for all vectors $w$, $\|w\|_{\infty}\leq \|w\|_p$, 
    \begin{align*}
    \sum_i  s^s|\AA(\xx-\yy)|_i^{s-2}(\AA\zz)_i^2 & = s^s \sum_i  |\AA(\xx-\yy)|_i^{s-p}|\AA(\xx-\yy)|_i^{p-2}(\AA\zz)_i^2 \\
     & \leq s^s\|\AA(\xx-\yy)\|^{s-p}_{\infty} \sum_i |\AA(\xx-\yy)|_i^{p-2}(\AA\zz)_i^2 \\
    & \leq s^s \|\AA(\xx-\yy)\|_p^{s-p} \|\zz\|_{\AA^{\top}\diag{|A(x-y)|^{p-2}}\AA}^2.
    \end{align*}
    Combining yields that,
    \begin{align}\label{eq:helpPrec}
        \zz^{\top}\nabla^2\ff(\xx)\zz & \leq  es(s-1)\zz^{\top}\AA^{\top}\diag{|\AA\yy-\bb|}^{s-2}\AA\zz + s^s \|\AA(\xx-\yy)\|_{p}^{s-p} \|\zz\|_{\AA^{\top}\diag{|A(x-y)|^{p-2}}\AA}^2 \notag\\
        & \leq e \|\zz\|_{\nabla^2\ff(\yy)} + s^s \|\AA(\xx-\yy)\|_{p}^{s-p} \|\zz\|_{\AA^{\top}\diag{|A(x-y)|^{p-2}}\AA}^2.
    \end{align}
    Now, let $\gg_{\yy}(\xx) = C_s\|\AA(\yy-\xx)\|_{p}^s$. By a simple calculation, we have that 
    \begin{align}\label{eq:helpPrec2}
    \nabla^2 \gg_y(\xx) = & C_s s(p-1)\|A(x-y)\|_p^{s-p}A^{\top}\diag{|A(x-y)|}^{p-2}A +C_s s(s-p)\|A(x-y)\|_p^{s-2p} \nonumber \\
    & A^{\top}\diag{|A(x-y)|}^{p-2}|A(x-y)||A(x-y)|^{\top}\diag{|A(x-y)|}^{p-2}A \nonumber\\
    & \succeq sC_s\|A(x-y)\|_p^{s-p}A^{\top}\diag{|A(x-y)|}^{p-2}A
    \end{align}
    We will now prove the upper and lower bounds on the Hessian of $f_y(x)$. We first show the upper bound $\nabla^2 \ff_{\yy}(\xx) \preceq e \cdot \nabla^2\hh_{\yy}(\xx)$.
    \begin{align*}
    \frac{1}{e}\nabla^2 \ff_{\yy}(\xx)  & = \frac{1}{e}\nabla^2\ff(\xx) +  \frac{1}{e}\nabla^2 \gg_{\yy}(\xx) \\
    & \preceq \nabla^2\ff(\yy) + \frac{s^s}{e}\|A(x-y)\|_p^{s-p}A^{\top}\diag{|A(x-y)|}^{p-2}A + \frac{1}{e}\nabla^2 \gg_{\yy}(\xx) && \text{(From Eq.~\eqref{eq:helpPrec})}\\
    & \preceq \nabla^2\ff(\yy) + \nabla^2 \gg_{\yy}(\xx) && \text{(From Eq.~\eqref{eq:helpPrec2} and $C_s\geq e\cdot s^s$)}\\ 
    & = \nabla^2\hh_{\yy}(\xx).
    \end{align*}
    For the lower bound, switching $\xx$ and $\yy$ in \eqref{eq:helpPrec} gives,
    \[
    \zz^{\top}\nabla^2\ff(\xx)\zz \geq \frac{1}{e}\|\zz\|_{\nabla^2\ff(\yy)} - \frac{s^s}{e} \|\AA(\xx-\yy)\|_{p}^{s-2} \|\zz\|_{\AA^{\top}\AA}^2.
    \]
    As a result,
    \[
    \nabla^2 \ff_{\yy}(\xx) = \nabla^2\ff(\xx) +  \nabla^2 \gg_{\yy}(\xx) \succeq \frac{1}{e}\cdot \left(\nabla^2\hh_{\yy}(\xx)- \nabla^2 \gg_{\yy}(\xx)\right) - \frac{1}{es}\nabla^2 \gg_{\yy}(\xx) + \nabla^2 \gg_{\yy}(\xx) \succeq \frac{1}{e}\cdot \nabla^2\hh_{\yy}(\xx).
    \]
\end{proof}

\noindent Similar to \cite{jambulapati2022improved}, we use the relative smoothness framework from~\cite{lu2018relatively}.

\begin{lemma}[Lemma 4.4 (Theorem 3.1 from~\cite{lu2018relatively})]\label{lem:relSmooth} Let $\ff, \hh$ be convex twice-differentiable functions satisfying
\[
\mu\nabla^2 \hh(\xx) \preceq \nabla^2 \ff(\xx) \preceq L \nabla^2 \hh(\xx)
\]
for all $\xx$. There is an algorithm which given a point $\xx_0$ computes a point $\xx$ with
\[
\ff(\xx)-\arg\min_{\yy}f(y) \leq \epsilon \left(\ff(\xx_0) - \arg\min_{\yy}f(y)\right)
\]
in $O\left(\frac{L}{\mu}\log\frac{1}{\epsilon}\right)$ iterations, where each iteration requires computing gradients of $\ff$ and $\hh$ at a point, $O(n)$ additional work, and solving a subproblem of the form
\[
\min\{\langle \gg,\xx\rangle + L \hh(\xx)\}
\]
for vectors $\gg$.
\end{lemma}
We will now prove Lemma~\ref{lem:proxSubProb}.
\begin{proof}[Proof of Lemma~\ref{lem:proxSubProb}]
    This follows from Lemma~\ref{lem:prec} and Lemma~\ref{lem:relSmooth}.

\end{proof}

We now state and prove our main reduction theorem.

\begin{proof}[Proof of Theorem~\ref{thm:regression-main}]

From Corollary~\ref{lem:ProxSNorm} after $k$ iterations, Algorithm~\ref{alg:ProxQ} finds $\xxtil$ such that for $\ff$ denoting the $\ell_s$-regression problem,
    \[
    \ff(\xxtil)-\ff(\xx^{\star})\leq \frac{s\|\AA\xx^{\star}-\AA\xx_0\|_p^{s}}{k^{\frac{s(p+1)-p}{p}}}.
    \]
    From Lemma~\ref{lem:RelateDivergenceFunc}, for for some $\xx_0$, $\|\AA\xx^{\star}-\AA\xx_0\|_p^p \leq 2^{p+p/s}n^{\frac{s-p}{s}}\left(\ff(\xx_0)-\ff(\xx^{\star})\right)^{p/s}$. Therefore,
    \[
     \ff(\xxtil)-\ff(\xx^{\star})\leq \frac{ s 2^{s+1}n^{\frac{s-p}{p}}}{k^{\frac{s(p+1)-p}{p}}} \left(\ff(\xx)-\ff(\xx^{\star})\right).
    \]
    The above implies that in $k = O\left(  n^{\frac{s-p}{s(p+1)-p}}\right)$ iterations, the function error reduces by $1/2$ and as a result in at most $ O\left(n^{\frac{s-p}{s(p+1)-p}}\log\frac{\ff(\xx_0)-\ff(\xx^{\star})}{\epsilon\ff(\xx^{\star})}\right)$ iterations, the function error reduces to $\epsilon\ff(\xx^{\star})$. In other words, we get an $\epsilon$-approximation to the $\ell_s$-regression problem. 
    
    Furthermore, from the relations between $s$ and $p$ norms, we can find a starting solution that is an $O\left(n^{s\left(\frac{1}{p}-\frac{1}{s}\right)}\right)$-approximate solution to the $\ell_s$-regression problem by solving one problem of the form of Eq~\eqref{eq:pNormOracle}. So in at most 
    \[
    O \left(s  n^{\frac{s-p}{s(p+1)-p}}\log \frac{n}{\epsilon}\right)
    \]
    iterations, each iteration solving a prox problem, we can find an $\epsilon$-approximation to the $\ell_s$-regression problem.

    Furthermore, from Lemma~\ref{lem:proxSubProb}, the proximal subproblem required to solve in each iteration can be reduced to solving $\Otil(1)$ problems of the form,
     \[
    \min_{\xx} \dd^{\top}\xx + \|\xx-\cc\|_{\nabla^2\ff(\cc)}^2 + \lambda_t \|\xx-\cc\|_p^p.
    \]
This concludes the proof of our result.
\end{proof}

\subsection{Proof of Theorem~\ref{thm:approxImpGMS}}\label{sec:proofApproxImpGMS}
\begin{proof}
We have 
    \begin{align*}
        \inner{\nabla f(\tilde{x})}{\tilde{x}-y} &= \inner{\nabla f(\tilde{x})}{\tilde{x}-x_{c,r}^*} + \inner{\nabla f(\tilde{x})-\nabla f(x_{c,r}^*)}{x_{c,r}^*-y} + \inner{\nabla f(x_{c,r}^*)}{x_{c,r}^*-y}\\
        &\leq \norm{\nabla f(\tilde{x})}_{p^*}\norm{\tilde{x}-x_{c,r}^*}_p + \norm{\nabla f(\tilde{x})-\nabla f(x_{c,r}^*)}_{p^*}\norm{x_{c,r}^*-y}_p + \inner{\nabla f(x_{c,r}^*)}{x_{c,r}^*-y},
    \end{align*}
    where we use Holder's inequality. Now, since $\tilde{x}$ is the output of $\ell_p^{\infty}(\delta,r)$ we have $\|\tilde{x}-x^*_{c,r}\|_p\leq \delta$ and $\|x^*_{c,r}-c\|_p\leq r$. Furthermore, since $f$ is smooth, 
    \[
    \norm{\nabla f(\tilde{x})-\nabla f(x_{c,r}^*)}_{p^*} \leq L\|\tilde{x}-x_{c,r}^*\|_p\leq L\delta.
    \]
    Using these in the above gives, 
    \[
     \inner{\nabla f(\tilde{x})}{\tilde{x}-y}  \leq \delta\norm{\nabla f(\tilde{x})}_{p^*} +  Lr\delta +\inner{\nabla f(x_{c,r}^*)}{x_{c,r}^*-y}.
    \]
    Since $x_{c,r}^*$ is the solution of a ball oracle, from Lemma~\ref{lem:HelpBall}, $\|x_{c,r}^*-\cc\|_p = r$ and 
    \[\nabla \ff(x_{c,r}^*) = -\frac{\|\nabla \ff(\xx_{c,r}^*)\|_{p/(p-1)}}{r^{p-1}} \diag{|\xx_{c,r}^*-\cc|}^{p-2}(\xx_{c,r}^*-\cc).
\]
As a result,
    \begin{align*}
       \inner{\nabla f(\tilde{x})}{\tilde{x}-y}  &\leq \delta\norm{\nabla f(\tilde{x})}_{p^*} +  Lr\delta - \frac{\norm{\nabla f(x_{c,r}^*)}_{p^*}}{r^{p-1}}\norm{x_{c,r}^*-y}_p^p\\
        &= \delta\norm{\nabla f(\tilde{x})}_{p^*} + L\delta r - \norm{\nabla f(x_{c,r}^*)}_{p^*}\norm{x_{c,r}^*-y}_p.
    \end{align*}
    Now, using the triangle inequality,
    \begin{align*}
     \inner{\nabla f(\tilde{x})}{\tilde{x}-y}     &\leq \delta\norm{\nabla f(\tilde{x})}_{p^*} + L\delta r - \norm{\nabla f(x_{c,r}^*)}_{p^*} \norm{\tilde{x} - y}_p \\
     &+ \norm{\nabla f(x_{c,r}^*)}_{p^*}\norm{x_{c,r}^*-\tilde{x}}_p \\
     & \leq \delta\norm{\nabla f(\tilde{x})}_{p^*} + L\delta r - \norm{\nabla f(x_{c,r}^*)}_{p^*} \norm{\tilde{x} - y}_p \\
   & +  \norm{\nabla f(x_{c,r}^*)-\nabla f(\tilde{x})}_{p^*}\norm{x_{c,r}^*-\tilde{x}}_p + \norm{\nabla f(\tilde{x})}_{p^*}\norm{x_{c,r}^*-\tilde{x}}_p && \text{(Triangle inequality)}\\
        &\leq 2\delta\norm{\nabla f(\tilde{x})}_{p^*} + L\delta (r+\delta) - \norm{\nabla f(x_{c,r}^*)}_{p^*} \norm{\tilde{x} - y}_p && \text{( $\|\tilde{x}-x^*_{c,r}\|_p\leq \delta$,$f$ is $L$ smooth)}\\ 
        &\leq 2\delta\norm{\nabla f(\tilde{x})}_{p^*} + L\delta (r+\delta) - \norm{\nabla f(\tilde{x})}_{p^*} \norm{\tilde{x} - y}_p\\
        &+ \norm{\nabla f(\tilde{x}) - \nabla f(x_{c,r}^*)}_{p^*} \norm{\tilde{x} - y}_p &&\text{(Triangle inequality)}\\
        &\leq 2\delta\norm{\nabla f(\tilde{x})}_{p^*} + 2L\delta (r+\delta) - \norm{\nabla f(\tilde{x})}_{p^*} \norm{\tilde{x} - y}_p\\
        &\leq - (1-\sigma)\norm{\nabla f(\tilde{x})}_{p^*} \norm{\tilde{x} - y}_p &&\text{(From our choice of $\delta$)}\\
        &= - (1-\sigma)\frac{\norm{\nabla f(\tilde{x})}_{p^*}}{\norm{\tilde{x} - y}_p^{p-1}} \norm{\tilde{x} - y}_p^p\\
        &= - (1-\sigma)\tilde{\gamma}\norm{\tilde{x} - y}_p^p.
    \end{align*}

    In addition, we have
    \begin{align*}
    \norm{\nabla f(\tilde{x})}_{p^*} = \frac{\norm{\nabla f(\tilde{x})}_{p^*}}{\norm{\tilde{x} - y}_p^{p-1}}\norm{\tilde{x} - y}_p^{p-1} = \tilde{\gamma} \norm{\tilde{x} - y}_p^{p-1} \leq (1+\sigma) \tilde{\gamma} \norm{\tilde{x} - y}_p^{p-1},
    \end{align*}
    and so it follows that the oracle implements a $\sigma$-$\textrm{GMS}_p$ oracle. We furthermore may observe that the oracle satisfies an $(\infty, (r-\delta)^{-1})$ movement bound, since
    \begin{equation*}
        \norm{\tilde{x}-y} \geq \norm{x_{c,r}^* - y} - \norm{x_{c,r}^* - \tilde{x}} \geq r - \delta.
    \end{equation*}
\end{proof}

\subsection{Proof of Theorem~\ref{thm:implementApproxOracle}}\label{sec:proofImpApproxOracle}
\begin{proof}
    The proof follows along similar lines as that for Theorem~\ref{thm:ConceptualProx}, though slightly modified to account for the smoothness in $\norm{\cdot}_{\nabla^2 f(c)}$, as well as the fact that the $x_{t+1}$ update is now constrained. In particular we have that, for all $t$, 
    \begin{align*}
        \inner{\nabla f(x_{t+1})}{x_{t+1} - y_t} &=  \inner{\nabla f(x_{t+1}) - \nabla f(x_t)}{x_{t+1} - y_t} + \inner{\nabla f(x_t)}{x_{t+1} - y_t}\\
        &\leq \norm{\nabla f(x_{t+1}) - \nabla f(x_t)}_{(\nabla^2 f(c))^{-1}}\norm{x_{t+1} - y_t}_{\nabla^2 f(c)} + \inner{\nabla f(y_t)}{x_{t+1} - y_t}\\
        & \leq \norm{\nabla f(x_{t+1}) - \nabla f(x_t)}_{(\nabla^2 f(c))^{-1}}\norm{x_{t+1} - y_t}_{\nabla^2 f(c)} - 2\xi \norm{x_{t+1} - y_t}_{\nabla^2 f(c)}^2\\
         & \leq - \xi \norm{x_{t+1} - y_t}_{\nabla^2 f(c)}^2,
    \end{align*}
    where the second inequality follows from the optimality conditions for $x_{t+1}$, and the final inequality follows from the smoothness in $\norm{\cdot}_{\nabla^2 f(c)}$.
    We additionally note that, for $d(x) = \norm{x}_{\nabla^2 f(c)}^2$, we have $\omega(x,y) = \norm{x-y}_{\nabla^2 f(c)}^2$, thus yielding the remaining parts of the proof, so that, letting $x_{k,0}$ and $x_{k,T}$ denote the starting and final iterates of the $k^{th}$ call to Algorithm~\ref{alg:genGD} (i.e., at the $k^{th}$ iteration of Algorithm~\ref{alg:restart}), we arrive at
    \begin{equation*}
        f(x_{k,T}) - f(x_{c,r}^*) \leq \frac{4\xi\norm{x_{k,0} - x_{c,r}^*}_{\nabla^2 f(c)}^2}{T^2}.
    \end{equation*}
Finally, the fact that, by Hessian stability, the function is $\xi^{-1}$ strongly convex with respect to $\norm{\cdot}_{\nabla^2 f(c)}$, when combined with a standard restarting argument, leads to
\begin{equation*}
    \norm{x_{K,T} - x_{c,r}^*}_p^2 \leq \frac{d^{1-\frac{2}{p}}}{\mu}\norm{x_K - x^*}_{\nabla^2 f(c)}^2 \leq \frac{\xi d^{1-\frac{2}{p}}}{\mu}(f(x_K) - f(x_{c,r}^*)) \leq \delta^2
\end{equation*}
    after $KT = O(\xi\log(\frac{\xi d(f(x_0) - f(x_{c,r}^*))}{\mu\delta})) \leq O(\xi\log(\frac{\xi dLr}{\mu\delta}))$ iterations, and so it follows that $\norm{x_{K,T} - x_{c,r}^*}_p \leq \delta$, as desired.
\end{proof}

\section{Proofs for Lower Bounds}\label{app:LB}
\subsection{Useful Lemmas}
We begin by observing useful properties of this function and its optimizer $\xx^\star$, whereby we find it helpful to define, for all $k'\leq k$, 
\[S_{k'} := \{\xx :  \supp{\xx} \subseteq \{1,\dots,k'\}\}.\]

\begin{lemma}\label{lem:CharSoln}
    Let $k' \leq k$. Then, for any $\xx\in S_{k'}$,
    \[
    \gg_k(\xx) \geq -(k'+1)\cdot \beta_k\alpha_k.
    \]
\end{lemma}
\begin{proof}
    Since $\xx\in S_{k'}$, we know that $\xx_{k'+1}=0$. Now,
    \[
    \gg_k(\xx) \geq \beta_k\max_{1\leq i \leq k}\{\xx_i - \alpha_k \cdot i\}\geq \beta_k (\xx_{k'+1} - \alpha_k \cdot (k'+1)) = - \beta_k\alpha_k \cdot (k'+1).
    \]
\end{proof}

\begin{lemma}\label{lem:CharOpt}
Let $\xx^{\star} $ denote the optimizer of Problem~\eqref{eq:HardProb}. Then for $1 \leq k'\leq k$,
\[
\gg_k(\xx^{\star}) \leq -\frac{\beta_k R}{{k'}^{1/p}}.
\]
In addition, we have that $\norm{\xx^{\star}} \leq 2R$.
\end{lemma}
\begin{proof}
    Consider the point $\yy_i = \begin{cases}
        -R/{k'}^{1/p} & \text{ if } i \leq k'\\
        0 & \text{otherwise.}
    \end{cases}$
    Now, $\|\yy\|_p \leq R$, and so it follows that
    \begin{align*}
    \gg_k(\xx^{\star})\leq \gg_k(\yy) = \max\left\{\ff_k(\yy), \beta_k(\norm{\yy}_p - 2R) - \alpha_k\right\} &\leq \max\left\{\frac{-\beta_k R}{{k'}^{1/p}} - \alpha_k, -\beta_k R - \alpha_k \right\} \\& \leq -\frac{\beta_k R}{{k'}^{1/p}}\ .
    \end{align*}

    In addition, we may observe that, for any $\xx$ s.t. $\norm{\xx}_p > 2R$, $\gg_k(\xx) \geq \beta_k(\norm{\xx}_p - 2R) - \alpha_k > -\alpha_k$, whereas $\gg_k(\xx^{\star}) \leq -\alpha_k$, meaning we have $\norm{\xx^{\star}}_p \leq 2R$.
\end{proof}
We now show that if there is a sequence of iterates $\xx^{(i)}$, for $1\leq i\leq k'$, such that $\xx^{(i)}\in S_{i}$, then we can give a lower bound.
\begin{lemma}\label{lem:LBFuncError}
    Let $\xx^{(i)}$, for $1\leq i\leq k'$ be such that $\xx^{(i)}\in S_{i}$ for all $i$. Then,
    \[
    \gg_k(\xx^{(i)})-\gg_k(\xx^{\star})\geq \frac{\beta_k R}{(i+1)^{1/p}} - (i+1)\cdot \beta_k\alpha_k,
    \]
    and so, after 
    $k$ iterations, we have 
    \[
    \gg_k(\xx_k)-\gg_k(\xx^{\star}) \geq  \frac{\lambda R^s}{16(k+1)^{\frac{s(p+1)-p}{p}}}.
    \]
\end{lemma}
\begin{proof}
    The proof follows from Lemma~\ref{lem:CharSoln} and Lemma~\ref{lem:CharOpt} and using $\left(1-\frac{1}{s}\right)^{s-1} \geq 1/4$.
\end{proof}
\subsection{Proof of Lemma~\ref{lem:Span}}
\LemSpan*
\begin{proof}
    We will prove this by induction. 
    \paragraph{Base case:} Let $\xx=0$, meaning that $\supp{\xx} = \emptyset$. 
    Now, letting $\xx' = \calO_{\gg_k,\lambda,p,s}(\xx) = \calO_{\gg_k,\lambda,p,s}(0)$,
    we must  have from optimality conditions one of the following two cases, depending on which function maximizes $\gg_k(\xx)$. (To simplify presentation, we let, $|\vv|$, $\mathrm{sgn}(\vv)$ denote the coordinate-wise absolute value and sign, respectively, for a vector $\vv$.)
     \begin{itemize}
         \item Case 1: \[\norm{\xx^{(i)}}^{1-p}(\xx^{(i)})^{p-1}= -\lambda\cdot s\|\xx^{(i)} - \cc\|_p^{s-p}|\xx^{(i)}- \cc|^{p-1}\cdot \mathrm{sgn}(\xx^{(i)}- \cc),\]
         where $(\xx^{(i)})^{p-1} := [(x_1^{(i)})^{p-1},\dots,(\xx_d^{(i)})^{p-1}]^\top$.
         \item Case 2: \[
    \nabla \ff_k(\xx^{(i)}) = -\lambda\cdot s\|\xx^{(i)} - \cc\|_p^{s-p}|\xx^{(i)}- \cc|^{p-1}\cdot \mathrm{sgn}(\xx^{(i)}- \cc).
    \]
     \end{itemize}

     For Case 1, it follows that $\xx' = 0 \in S_{1}$, and so the zero-respecting condition is satisfied.
It remains to handle Case 2, which occurs when the maximizing function is $\ff_k(\xx)$.
   Now, let $\xx' = -\left(\frac{\beta_k}{s\lambda}\right)^{\frac{1}{s-1}}\ee_1$. We claim that $\xx'$ satisfies the optimality conditions.
    First, from the definition of $\gg_k$, we have that 
    \[
    \gg_k(\xx') = \beta_k \max_{1\leq i\leq k}\left(\xx'_i - i\cdot \alpha_k\right).
    \]
  Now, from the value of $\alpha_k$, since $\xx'_1 - \alpha_k > -2\alpha_k >  -3\alpha_k,\cdots$, we must have that $\gg_k(\xx') = \beta_k (\xx_1-\alpha_k)$. Therefore, $\nabla\gg_k(\xx') = \beta_k \ee_1$. Thus, one may verify that the optimality conditions are satisfied by $\xx'$, which proves the base case since $\supp{\xx'} = \{1\}$.

     \paragraph{If $\xx^{(j)}\in S_j, \forall j\leq i-1$:} Since $\xx^{(i)} =\calO_{\gg_k,\lambda,p,s}(\cc)$, $\cc\in S_{i-1}$, we must again have from optimality conditions one of the following two cases, depending on which function maximizes $\gg_k(\xx)$.
     \begin{itemize}
         \item Case 1: \[\norm{\xx^{(i)}}^{1-p}(\xx^{(i)})^{p-1}= -\lambda\cdot s\|\xx^{(i)} - \cc\|_p^{s-p}|\xx^{(i)}- \cc|^{p-1}\cdot \mathrm{sgn}(\xx^{(i)}- \cc),\]
         where $(\xx^{(i)})^{p-1} := [(x_1^{(i)})^{p-1},\dots,(\xx_d^{(i)})^{p-1}]^\top$.
         \item Case 2: \[
    \nabla \ff_k(\xx^{(i)}) = -\lambda\cdot s\|\xx^{(i)} - \cc\|_p^{s-p}|\xx^{(i)}- \cc|^{p-1}\cdot \mathrm{sgn}(\xx^{(i)}- \cc).
    \]
     \end{itemize}

     For Case 1, it follows that, since $\cc \in S_{i-1}$, we have that $\xx^{(i)} \in S_{i-1} \subseteq S_{i}$, and so the zero-respecting condition is satisfied.

     It remains to handle Case 2, which occurs when the maximizing function is $\ff_k(\xx)$. To begin, consider $\xx^{(i)} = \cc - \gamma \ee_i$ for $\gamma = \left(\frac{\beta_k}{s\lambda}\right)^{1/(s-1)}$.
Since, by our choice of $\gamma$ (as well as $\alpha_k$), we have that $-\gamma - i \alpha_k > -(i+1)\alpha_k > -(i+2)\alpha_k \cdots$, it follows that $\arg\max_j \beta_k (\xx_j^{(i)}-j\alpha ) \subseteq \{1,\cdots, i\}$. Therefore, we have that $\nabla\gg_k(\xx^{(i)}) \in \beta_k \mathrm{conv}(\ee_j, j\leq i)$. Furthermore, we observe that
\[
-\lambda\cdot s\|\xx^{(i)} - \cc\|_p^{s-p}|\xx^{(i)}- \cc|^{p-1}\cdot \mathrm{sgn}(\xx^{(i)}- \cc) = \beta \ee_i \in \beta_k \mathrm{conv}(\ee_j, j\leq i),
\]
which establishes that $x^{(i)}$ satisfies the optimality conditions. It follows that $\supp{\xx^{(i)}} \subseteq \{1,2,\cdots,i\} $.
\end{proof}

\subsection{Proof of Theorem~\ref{thm:LowerBound-main}}\label{app:LBsTheorem}
\LowerBound*
\begin{proof}
Let $\calA$ be any $\ell_p^s(\lambda)$-proximal zero-respecting algorithm and construct $\ff = \gg_k$ as in \eqref{eq:HardProb}. Now, consider the optimization problem $\min_{\|\xx\|_p\leq R}\ff(\xx)$. From Lemma~\ref{lem:Span}, the first $k$ queries $\xx^{(i)}, i\leq k$ from $\calA$, satisfy $\xx^{(i)}\in S_i$ for all $i\leq k$. Now, from Lemma~\ref{lem:LBFuncError} for all $i\leq k$,
\[
\ff(\xx^{(i)})-\ff(\xx^{\star})\geq  \frac{\lambda R^s}{ 16(i+1)^{\frac{s(p+1)-p}{p}}}.
\]
Thus, for all $i \leq k = O\left(\frac{\lambda R^s}{\epsilon} \right)^{\frac{1}{s(1+\nu)}}$, $\ff(\xx^{(i)})-\ff(\xx^{\star})\geq \epsilon.$
\end{proof}

\subsection{Proof of Lemma~\ref{lem:SpanInf}}
\begin{restatable}{lemma}{LemSpanInf}\label{lem:SpanInf}
Let $k, d \in \nat$ be such that $1 \leq k\leq d$, and let $\gg_k$ be as in \eqref{eq:HardProbInf}. Then, $\gg_k$ is an $\ell_p^{\infty}(r)$-proximal zero-chain.
\end{restatable}
\begin{proof}
    We prove by induction. 
    \paragraph{Base case:} We know that $\xx^{(0)}=0$. Now, since $\xx^{(1)} = \calO_{\gg_k,r,p,\infty}(0)$, we must have $\|\xx^{(1)}\|_p\leq r$ and using Lagrange duality, we want to solve,
    \[
    \min_{\xx}\max_{\mu} \gg_1(\xx) +\mu(\|\xx\|_p^p-r^p).
    \]
    When $\mu> 0$, we must have $\|\xx^{(1)}\|_p = r$. If $\mu = 0$ then $\nabla\gg_k(\xx^{(1)}) = 0$ and $\|\xx^{(1)}\|_p <r$. The later case cannot happen since $\nabla\gg_k(\xx)$ is not 0 for any $\xx$. Therefore, $\xx^{(1)} = -r\ee_1$ satisfies the above optimality conditions if the corresponding $\mu$ is positive. If that is true, then since the problem is strictly convex $\xx^{(1)}$ must be the unique solution.

   We now verify that for this value of $\xx^{(1)}$, $\mu >0$. This is because, $\gg_k(\xx^{(1)}) = \xx^{(1)}-\alpha$ since $\alpha \geq 4r$. As a result $\nabla\gg_k(\xx^{(1)}) = \ee_1$. From the optimality conditions, $\nabla\gg(\xx^{(1)}) = - \mu p |\xx^{(1)}|^{p-2}\xx^{(1)}$ which gives, $\mu = 1/(pr^{p-1}) >0$.
     \paragraph{If $\xx^{(j)}\in S_j, \forall j\leq i-1$:} Since $\xx^{i} =\calO_{\gg_k,r,p,\infty}(\cc)$, we must have that $\|\xx^{(i)}-\cc\|_p =r$
    and following a similar reasoning with respect to the optimality conditions as before, it again follows that $\xx^{(i)} = \cc - r\ee_i \in S_i$.
\end{proof}

\subsection{Proof of Theorem~\ref{thm:LowerBoundInf-main}}\label{app:LBinfThm}
Similar to the $s < \infty$ setting, and letting $\ff_k(\xx) = \max_{1\leq i\leq k}\{\xx_i - i \cdot 4r\}$, we consider the following hard function:
\begin{equation}\label{eq:HardProbInf}
\gg_k(\xx) := \max\left\{\ff_k(\xx), \norm{\xx}_p - 2R - 4r\right\}.
\end{equation}
Now from Lemma~\ref{lem:LBFuncError}, for $\xx^{(i)}\in S_i$ $(i \geq 1)$,
\begin{equation}\label{eq:LBInf}
\gg_k(\xx^{(k)})-\gg_k(\xx^{\star}) \geq \frac{R}{(i+1)^{1/p}} - 4r(i+1).
\end{equation}

We can now complete the proof of Theorem~\ref{thm:LowerBoundInf-main}, which we restate here for convenience.
\LowerBoundInf*
\begin{proof}
From Lemma~\ref{lem:SpanInf}, $\xx^{(i)}\in S_i$. Further, from Eq.~\eqref{eq:LBInf},
\[
\gg_k(\xx^{(i)})-\gg(\xx^{\star})\geq \frac{R}{(i+1)^{1/p}} - 4r(i+1).
\]
So for the first $k = O((R/r)^{\frac{p}{p+1}})$ iterations, $\gg_k(\xx^{(i)})-\gg(\xx^{\star})\geq \Omega(R^{1/(p+1)}r^{p/(p+1)})$.
\end{proof}

\end{document}